\documentclass{amsart}
\usepackage{graphicx,amsfonts,amssymb,amsmath,amsthm,url}

\theoremstyle{plain} 
\newtheorem{theorem}             {Theorem}  [section]
\newtheorem{lemma}      [theorem]{Lemma}

\newtheorem{corollary}  [theorem]{Corollary}
\newtheorem{proposition}[theorem]{Proposition}

\newtheorem{conjecture} [theorem]{Conjecture}

\theoremstyle{definition}
\newtheorem{definition} [theorem]{Definition}

\theoremstyle{remark}
\newtheorem{remark} [theorem]             {Remark}

\def\sgn{\operatorname{sgn}}
\def\Stab{\operatorname{Stab}}
\def\PSL{\operatorname{PSL}}

\def\res{\operatorname{res}}
\def\vol{\operatorname{vol}}
\def\SO{\operatorname{SO}}
\def\PGL{\operatorname{PGL}}
\def\GL{\operatorname{GL}}
\def\ad{\operatorname{ad}}
\def\SL{\operatorname{SL}}
\renewcommand{\Re}{\mathrm{Re}}
\renewcommand{\Im}{\mathrm{Im}}
\def\eps{\varepsilon}

\begin{document}

\title{Equidistribution of cusp forms in the level aspect}

\author{Paul D. Nelson}
\begin{abstract}
  Let $f$ traverse a sequence of classical holomorphic newforms of fixed
  weight and increasing squarefree level $q \rightarrow \infty$.
  We prove that
  the pushforward of the mass of $f$ to the modular curve of
  level $1$ equidistributes with respect to the Poincar\'{e}
  measure.

  Our result answers affirmatively the squarefree level case of
  a conjecture spelled out in 2002 by Kowalski, Michel, and
  VanderKam
  \cite{KMV02}
  in the spirit of a conjecture of Rudnick and Sarnak \cite{MR1266075}
  made in 1994.

  Our proof follows the strategy of Holowinsky and Soundararajan
  \cite{MR2680499} who showed in 2008 that
  newforms of level $1$ and large weight have equidistributed
  mass.  The new ingredients required to treat forms of fixed
  weight and large level are an adaptation of Holowinsky's
  reduction of the problem to one of bounding shifted sums of
  Fourier coefficients, a refinement of his bounds for shifted
  sums, an evaluation of the $p$-adic integral needed to extend
  Watson's formula to the case of three newforms where the level
  of one divides but need not equal the common squarefree level
  of the other two, and some additional technical work in the
  problematic case that the level has many small prime factors.
\end{abstract}
\maketitle

\tableofcontents

\newpage
\section{Introduction}
\subsection{Statement of result}
A basic problem in modern number theory and the analytic
theory of modular forms is to understand the limiting
behavior of modular forms in families.
Let $f : \mathbb{H} \rightarrow \mathbb{C}$ be a classical
holomorphic
newform of weight $k$ and level $q$.
The \emph{mass} of $f$ is the finite measure
$d \nu_f = |f(z)|^2 y^{k-2} \, d x
\, d y$ ($z = x + i y$) on the modular curve $Y_0(q) =
\Gamma_0(q) \backslash \mathbb{H}$.
In a recent breakthrough, Holowinsky and Soundararajan
\cite{MR2680499}
proved that newforms of large weight $k$ and fixed level $q = 1$
have equidistributed mass,
answering affirmatively a natural variant\footnote{as spelled out by Luo
  and Sarnak \cite{luo-sarnak-mass}; we refer to Sarnak
  \cite{MR1321639,sarnak-progress-que} and the references in
  \cite{MR2680499} for further discussion.}
of the \emph{quantum unique ergodicity} conjecture of Rudnick
and Sarnak \cite{MR1266075}.
\begin{theorem}[Mass equidistribution for $\SL(2,\mathbb{Z})$
  in the weight aspect]\label{thm:hque-weight}
  Let $f$ traverse a sequence of newforms of increasing weight
  $k \rightarrow \infty$ and fixed level $q = 1$.  Then the mass $\nu_f$
  equidistributes\footnote{We say that a sequence of finite Radon
    measures $\mu_j$ on a locally compact Hausdorff space $X$
    \emph{equidistributes} with respect to some fixed finite
    Radon measure $\mu$ if for each function $\phi \in C_c(X)$
    we have $\mu_j(\phi) / \mu_j(1) \rightarrow \mu(\phi) /
    \mu(1)$ as $j \rightarrow \infty$, here and always
    identifying a measure $\mu$ with the corresponding linear
    functional $\phi \mapsto \mu(\phi) := \int_X \phi \, d \mu$
    on the space $C_c(X)$
    and writing $1$ for the constant function.  } with respect to
  the Poincar\'{e} measure $d \mu = y^{-2} \, d x \, d y$ on
  the modular curve $Y_0(q)$.
\end{theorem}
Kowalski, Michel, and VanderKam \cite[Conj 1.5]{KMV02} formulated
an analogue of the Rudnick-Sarnak conjecture in which the roles
of the parameters $k$ and $q$ are reversed: they conjectured
that the masses of newforms of fixed weight and large level $q$
are equidistributed amongst the fibers of the canonical
projection $\pi_q : Y_0(q) \rightarrow Y_0(1)$
in the following sense.
\begin{conjecture}[Mass equidistribution for $\SL(2,\mathbb{Z})$ in
  the level aspect]\label{conj:1}
  Let $f$ traverse a sequence of newforms of fixed weight and
  increasing level $q \rightarrow \infty$.  Then the pushforward $\mu_f := \pi_{q*}(\nu_f)$
  of the mass of $f$ to $Y_0(1)$ equidistributes with respect to
  $\mu$.
\end{conjecture}
Kowalski, Michel, and VanderKam
remark that
Conjecture \ref{conj:1} follows in the special case of dihedral
forms from their subconvex bounds for Rankin-Selberg
$L$-functions modulo an unestablished extension of Watson's
formula \cite{watson-2008}, which is now known by Theorem
\ref{thm:watson-ext} of this paper.  Recently Koyama
\cite{MR2511663}, following the method of Luo and Sarnak
\cite{MR1361757}, proved the analogue of Conjecture \ref{conj:1}
for unitary Eisenstein series of increasing prime level by
reducing the problem to known subconvex bounds for automorphic
$L$-functions of degree two.

Our aim in this paper is to establish
the squarefree level case of
Conjecture \ref{conj:1}.
Our result is the first of its
kind for nondihedral cusp forms.
\begin{theorem}[Mass equidistribution for
  $\SL(2,\mathbb{Z})$ in the squarefree level
  aspect]\label{thm:1}
  Let $f$ traverse a sequence
  of newforms of fixed weight and increasing
  squarefree level $q \rightarrow \infty$.  Then $\mu_f$
  equidistributes with respect to $\mu$.
\end{theorem}


\begin{remark}\label{rmk:1}
  Our extension (Theorem \ref{thm:watson-ext}) of Watson's
  formula \cite{watson-2008} shows that Theorem \ref{thm:1}
  would follow from subconvex bounds
  $L(f \times f \times \phi,1/2) \ll_{\phi} q^{1-\delta}$ ($\delta >
  0$)
  for the central
  $L$-values  of the triple
  product $L$-functions attached to $f$ as above and each
  Maass cusp form or unitary Eisenstein series $\phi$ on
  $Y_0(1)$.  Such bounds are known to follow from the
  generalized Lindel\"{o}f hypothesis, which itself follows
  from the generalized Riemann hypothesis, so one can view
  Theorem \ref{thm:1} as an unconditionally proven consequence of
  a central unresolved conjecture.
\end{remark}
\begin{remark}
  One cannot relax entirely the restriction of Theorem
  \ref{thm:1} to newforms, since
  for instance a cusp form of level $1$ may
  be regarded as an oldform of arbitrary level $q > 1$.
\end{remark}
\begin{remark}
  Rudnick \cite{MR2181743} showed that Theorem
  \ref{thm:hque-weight} implies that the zeros of newforms of
  level $1$ and weight $k \rightarrow \infty$ equidistribute on
  $Y_0(1)$.
  At the 2010 Arizona Winter School,
  Soundararajan asked whether
  there is an analogue of Rudnick's result for newforms of large
  level.
  We do not know whether such an analogue exists and
  highlight here one of the difficulties in adapting
  Rudnick's method.  Let $f$ be a newform of weight $k$ and
  level $q$, let $\mathcal{Z}$ be the left
  $\Gamma_0(q)$-multiset of zeros of $f$ in $\mathbb{H}$ and let
  $\mathcal{Z}_1$ be the left $\Gamma$-multiset ($\Gamma =
  \PSL(2,\mathbb{Z})$) obtained by summing the images of
  $\mathcal{Z}$ under coset representatives for $\Gamma(1) /
  \Gamma_0(q)$.  We ask: does $\Gamma \backslash
  \mathcal{Z}_1$ equidistribute on $Y_0(1)$ as $q \rightarrow
  \infty$?  Following Rudnick, one may show for $\phi \in
  C_c^\infty(\mathbb{H})$ and $\Phi(z) = \sum_{\gamma \in
    \Gamma} \phi(\gamma z)$ that
  \begin{equation}\label{eq:26}
    \frac{12}{k \psi(q)}
    \sum_{z \in \Gamma \backslash \mathcal{Z}_1}
    \frac{ \Phi(z) }{\# \Stab_\Gamma(z)}
    = 
    \int_{\Gamma \backslash \mathbb{H}} \Phi \, d V \\
    +
    \int _{\Gamma \backslash \mathbb{H} }
    \frac{\pi_{q*}(\log \nu_f)}{k \psi(q)}
    \Delta \Phi \, d V,
  \end{equation}
  where $\psi(q) = [\Gamma(1):\Gamma_0(q)]$, $\Delta = y^2
  (\partial_x^2 + \partial_y^2)$ is the hyperbolic Laplacian,
  and $d V$ is the hyperbolic probability measure on $\Gamma
  \backslash \mathbb{H}$; the formula \eqref{eq:26} follows by
  some elementary manipulations of the identity
  $\int_{\mathbb{H}} \log |z-z_0| \Delta \phi(z) y^{-2} \, d x
  \, d y = 2 \pi \phi(z_0)$, which holds for any $z_0 \in
  \mathbb{H}$ and follows from Green's identities.  Since the
  total number of inequivalent zeros is $\# \Gamma \backslash
  \mathcal{Z}_1 = \# \Gamma_0(q) \backslash \mathcal{Z} \sim k
  \psi(q)/12$ \cite[\S 2]{MR0314766}, the first term on the
  right-hand side of \eqref{eq:26} may be regarded as a main
  term, the second as an error term that one would like to show
  tends to $0$.  An important step toward adapting Rudnick's
  method would be to rule out the possibility that
  $\pi_{q*}(\log \nu_f) / k \psi(q)$ tends to $-\infty$
  uniformly on compact subsets as $q \rightarrow \infty$.  The
  difficulty in doing so is that Theorem \ref{thm:1} does not
  seem to preclude the masses $\nu_f$ from being very small
  somewhere within each fiber of the projection $Y_0(q)
  \rightarrow Y_0(1)$; stated another way, the \emph{sum} of the
  values taken by $y^k |f|^2$ in a fiber of $Y_0(q) \rightarrow
  Y_0(1)$ are controlled (in an average sense
  as the fiber varies) by Theorem
  \ref{thm:1}, but their \emph{product} could still conceivably
  be quite small.  There are further difficulties in adapting
  Rudnick's method that we shall not mention here.
\end{remark}

\begin{remark}
  Lindenstrauss \cite{MR2195133} and Soundararajan
  \cite{MR2680500} proved that Maass
  eigencuspforms of fixed level $q$ and large Laplace eigenvalue
  $\lambda \rightarrow \infty$ have equidistributed mass.  We
  ask: do Maass newforms of large level $q \rightarrow \infty $
  (with $\lambda$ taken to lie in a fixed subinterval of
  $[1/4,+\infty]$, say) satisfy the natural analogue of
  Conjecture \ref{conj:1}?  An affirmative
  answer to this question would follow from the
  generalized Riemann hypothesis (at least for $q$
  squarefree, as in remark \ref{rmk:1}), but appears beyond the
  reach of our methods because the Ramanujan conjecture is not
  known for Maass forms (compare with
  \cite[p.2]{MR2680499}).
\end{remark}

\begin{remark}
  We shall actually establish the following
  stronger hybrid equidistribution result: for a newform $f$ of
  (possibly \emph{varying}) weight $k$ and squarefree level $q$,
  the measures $\mu_f = \pi_{q*}(\nu_f)$ equidistribute as $q k
  \rightarrow \infty$.
  The novelty in our argument concerns only the variation of
  $q$, so we encourage the reader to regard $k$ as fixed.
\end{remark}

\begin{remark}
  With minor modifications our arguments should extend to the
  general case of not necessarily squarefree levels $q$ as soon
  as an appropriate extension of Watson's formula is worked out.
  However, we shall invoke the assumption that the level $q$ is
  squarefree whenever doing so simplifies the exposition.
  The
  parts of our argument that require modification to treat the
  general case are
  Lemmas \ref{lem:expand-int-e-of-h}, \ref{lem:stupid-sums}, and
  \ref{lem:p-adic-integral}.  One should be able to generalize
  Lemmas \ref{lem:expand-int-e-of-h} and \ref{lem:stupid-sums}
  using that for any level $q$ the cusps of $\Gamma_0(q)$ fall
  into classes indexed by the divisors $d$ of $q$ consisting of
  $\phi(\gcd(d,q/d))$ cusps of width $d / \gcd(d,q/d)$.  To
  generalize \ref{lem:p-adic-integral}, one must compute
  (or sharply bound) a
  $p$-adic integral involving matrix coefficients of
  supercuspidal representations of $\GL(2,\mathbb{Q}_p)$.
  We plan to consider this generalization in future work.
\end{remark}

\subsection{A very brief review
  of the motivating work of Holowinsky-Soundararajan}
Our proof of Theorem \ref{thm:1}
is an adaptation
of the Holowinsky-Soundararajan proof
\cite{MR2680499} of Theorem
\ref{thm:hque-weight},
which in turn synthesizes the independent
arguments of Holowinsky \cite{MR2680498} and Soundararajan
\cite{MR2680497}.
Here we briefly recall their independent
arguments and refer to the survey
\cite{sarnak-progress-que} and the original papers for further
background.

Holowinsky \cite{MR2680498} employs a clever unfolding trick and an asymptotic
analysis of certain archimedean integrals in the limit $k
\rightarrow \infty$ to reduce the study of the periods
$\mu_f(\phi)$ to the problem of bounding sums roughly of the
form
\begin{equation}\label{eq:32}
  \sum_{n \ll  k} \lambda_f(n) \lambda_f(n +l),
\end{equation}
where $l$ is an essentially bounded nonzero integer and
$\lambda_f(n)$ the $n$th Fourier coefficient of the newform $f$
of weight $k \rightarrow \infty$ and level $1$, normalized so
that the Deligne bound reads $|\lambda_f(p)| \leq 2$.  He
reduces bounds for the sums \eqref{eq:32}
to those for the mean values $\sum_{n \leq k} |\lambda_f(n)|$
by a sieving technique that quantifies,
using the Deligne bound for $\lambda_f$,
the ``independence'' of the maps
$n \mapsto \lambda_f(n)$ and $n \mapsto \lambda_f(n+l)$ for $l
\neq 0$.

Soundararajan's method
\cite{MR2680497}
takes as input a precise identity (given in this case by Watson
\cite{watson-2008}) relating the
Weyl periods $\mu_f(\phi)$ for the equidistribution problem
(here $\phi$ is a Maass form of level $1$) to the central value
of the triple product $L$-function $L(\phi \times f \times
f,\tfrac{1}{2})$, which
he then bounds by a method that applies
more generally to any $L$-function $L(s,\pi)$ satisfying certain
hypotheses  that are implied by the generalized Ramanujan conjectures
when $\pi$ is automorphic.

\subsection{What's new in this paper}
The synthetic part of the Holowinsky-Soundararajan argument
works just as well in the level aspect as in the weight aspect
(see \S5), so we highlight here four of the more substantial difficulties
encountered in adapting the independent arguments of Holowinsky
and Soundararajan to the level aspect.



\begin{enumerate}
\item It is not a priori clear how best to extend Holowinsky's
  unfolding trick in the presence of multiple (possibly
  unboundedly many) cusps, nor what should take the place of his
  asymptotic analysis of archimedean integrals in studying the
  fixed weight, large level limit; several fundamentally
  different approaches are possible, one of which we shall
  present in \S\ref{sec:reduct-shift-sums}.

  When $q$ is squarefree,
  the problem then becomes to bound sums roughly of the
  form\footnote{Here one should think of a divisor $d$ of $q$ as indexing
    the unique cusp of $\Gamma_0(q)$ of width $d$,
    where as usual
    the \emph{width} of a cusp is its ramification index
    over the cusp $\infty$ for $\Gamma_0(1)$.}
  \begin{equation}\label{eq:30}
    \sum_{d|q}
    \sum_{n \ll  d k} \lambda_f(n) \lambda_f(n+ d l),
  \end{equation}
  where again $l \neq 0$ is essentially bounded.
  As we now explain, the sums \eqref{eq:30}
  differ from
  those \eqref{eq:32} studied by Holowinsky
  in two important ways.

\item The shifts $d l$ are now nearly as large as the length of
  the interval $\approx d k$ over which we are
  summing.\footnote{This difficulty corresponds the fact that
    cusps for $\Gamma_0(q)$ may have \emph{large} width.}  Much
  of the existing work on bounds for such sums (see remark
  \ref{rmk:bounds-indiv-shift}) applies only when the shift is
  substantially smaller than the summation interval.
  Holowinsky's treatment of \eqref{eq:32} \emph{does} allow
  shifts as large as the summation interval, but gives a bound
  for $\sum_{n \ll q k} \lambda_f(n) \lambda_f(n + q l)$ that
  involves an extraneous factor of $\tau(q l)$, which is
  prohibitively large (e.g., $\gg \log(q)^A$ for any $A$) if
  $q$ has many small prime factors.  In Theorem
  \ref{thm:holow-sums-2}, we refine Holowinsky's method to
  allows shifts as large as the summation interval
  with \emph{full uniformity} in the size of the shift,
  e.g., without the factor $\tau(q l)$.
  This refinement may be of independent interest.
\item Let $\omega(q)$ denote the number of prime divisors of the
  squarefree integer $q$.
  Then the number of shifted sums in \eqref{eq:30}
  is $2^{\omega(q)}$,
  which can be quite large.\footnote{This difficulty corresponds
    to the fact that $\Gamma_0(q)$ may have \emph{many} cusps.}
  In the crucial case\footnote{Soundarajan's argument succeeds
    unless this is so.} that $|\lambda_f(p)|$ is typically
  small for primes $p \ll q k$, our refinement of Holowinsky's
  method saves nearly two logarithmic powers of $d k$ over the
  trivial bound $\ll d k$ for the shifted sum in \eqref{eq:30}
  of length $\approx d k$.
  Thus we save very little over the
  trivial bound if $d$ is a small divisor of $q$, and it is not
  immediately clear whether such savings are enough to
  produce a sufficient saving in the sum over all $d$.  One needs here an
  inequality of the shape
  \begin{equation}\label{eq:33}
    \sum_{d | q} \frac{d k}{\log(d k)^{2- \eps}}
    \ll
    \frac{q k}{\log(q k)^{2 - \eps}}
    \log \log (e^e q),
  \end{equation}
  which one can interpret as saying that the divisors of
  any squarefree
  integer are well distributed in a certain sense.  Indeed, if
  hypothetically $q$ were to have ``too many'' large divisors,
  then the LHS of \eqref{eq:33} might be large enough
  to swamp the
  small logarithmic savings, while if $q$ were to
  have ``too many'' small divisors, then the savings for each
  term on the LHS might be too small to produce an overall
  savings.  A convexity argument and a (weak form of the) prime
  number theorem are sufficient to establish \eqref{eq:33}; see
  Lemma \ref{lem:stupid-sums}.
\item
  The identity relating $\mu_f(\phi)$
  to $L(\phi \times f \times f, \tfrac{1}{2})$
  that Soundararajan's method takes as input
  is given by Watson
  \cite{watson-2008}
  when $f$ and $\phi$ are newforms of the \emph{same}
  (squarefree)
  level.
  In the level aspect,
  the relevant
  Weyl periods are those for which $f$ has large level
  and $\phi$ has fixed level,
  so Watson's formula does not apply.
  We extend Watson's
  result in Theorem \ref{thm:watson-ext}
  by computing (Lemma
  \ref{lem:p-adic-integral})
  a $p$-adic integral arising
  in Ichino's general formula \cite{MR2449948},
  specifically
  \begin{equation}\label{eq:34}
    \int_{g \in \PGL_2(\mathbb{Q}_p)}
    \frac{
      \langle g \cdot \phi_p, \phi_p \rangle
    }{
      \langle \phi_p, \phi_p \rangle
    }
    \frac{
      \langle g \cdot f_p, f_p \rangle
    }{
      \langle f_p, f_p \rangle
    }
    \frac{
      \langle g \cdot f_p, f_p \rangle
    }{
      \langle f_p, f_p \rangle
    }
    \, d g,
  \end{equation}
  where $\phi_p$ (resp$.$ $f_p$) is the newvector at $p$ for the
  adelization of $\phi$ (resp$.$ $f$) and $\langle , \rangle$
  denotes a $\PGL(2,\mathbb{Q}_p)$-invariant Hermitian pairing on
  the appropriate representation space.  The crucial case for us
  is when $p$ divides the squarefree level $q$ of the newform $f$,
  so that $\phi_p$ lives in a spherical representation of
  $\PGL_2(\mathbb{Q}_p)$ and $f_p$ in a special representation.
  As we discuss in remark
  \ref{rmk:an-extension-watsons-1},
  our evaluation of \eqref{eq:34} leads to a precise
  formula relating $\int \psi_1 \psi_2 \psi_3$ to
  $L(\tfrac{1}{2},\psi_1 \times \psi_2 \times \psi_3)$ for
  \emph{any} three newforms of squarefree level
  (and trivial central character); such an identity
  should be of general use in
  future work that exploits the connection
  between periods and $L$-values.
\end{enumerate}

\subsection{Plan for the paper}\label{sec:outline-proof}
Our paper is organized as follows.  In \S
\ref{sec:modular-forms-their} we recall some standard properties
of our basic objects of study: holomorphic newforms, Maass
eigencuspforms, unitary Eisenstein series and incomplete
Eisenstein series.
In \S \ref{sec:main-estimates} we prove the level aspect
analogue of Holowinsky's main result \cite[Corollary
3]{MR2680498}, as described above;
we emphasize the aspects of his argument
that do not immediately generalize
to the level aspect and refer to his paper
for the details of arguments that do.
In \S \ref{sec:an-extension-watsons} we extend Watson's formula
to cover the additional case that we need.
In \S \ref{sec:proof-theor-refthm:1} we
complete the proof of Theorem \ref{thm:1} using the main results
of \S \ref{sec:main-estimates} and
\S \ref{sec:an-extension-watsons}.
Sections
\ref{sec:main-estimates} and \ref{sec:an-extension-watsons} are
independent of each other, but both depend upon the
definitions, notation and facts recalled in \S
\ref{sec:modular-forms-their}.

\subsection{Notation and conventions}\label{sec:notation-conventions}
Recall the standard notation
for the upper half-plane
$\mathbb{H} = \{z \in \mathbb{C} : \Im(z) > 0\}$, the modular
group $\Gamma = \SL(2,\mathbb{Z}) \circlearrowright \mathbb{H}$
acting by fractional linear transformations,
its congruence subgroup $\Gamma_0(q)$
consisting of those elements with lower-left entry
divisible by $q$,
the modular curve $Y_0(q) = \Gamma_0(q) \backslash \mathbb{H}$,
the natural projection $\pi_q : Y_0(q)
\rightarrow Y_0(1)$,
the Poincar\'{e} measure $d \mu = y^{-2} \, d x \, d y$,
and the stabilizer  $\Gamma_\infty = \{\pm \left(
  \begin{smallmatrix}
    1&n\\
    &1
  \end{smallmatrix}
\right) : n \in \mathbb{Z} \}$ in $\Gamma$ of $\infty
\in \mathbb{P}^1(\mathbb{R})$.
We denote a typical element of $\mathbb{H}$ as
$z = x + i y $ with $x, y \in \mathbb{R} $.


There is a natural inclusion
$C_c(Y_0(1)) \hookrightarrow C_c(Y_0(q))$
obtained by pulling back under the projection $\pi_q$; here $C_c$ denotes the space of compactly supported
continuous functions.
For a newform $f$ of weight $k$ on $\Gamma_0(q)$ the
pushforward measure $d \mu_f := \pi_{q*}(|f|^2 y^k \, d \mu)$ on
the modular curve $Y_0(1)$
corresponds, by definition, to the linear functional
\[
\mu_f(\phi) = 
\int_{\Gamma_0(q) \backslash \mathbb{H}}
\phi(z) |f|^2(z) y^k \, \frac{d x \, d y}{y^2}
\quad \text{ for } \phi \in C_c(Y_0(1)) \hookrightarrow C_c(Y_0(q)).
\]
We let $\mu$ denote the standard
measure on $Y_0(1)$, so that
\[
\mu(\phi)
= \int_{\Gamma \backslash \mathbb{H}}
\phi(z) \, \frac{d x \, d y }{y^2 }
\quad \text{ for } \phi \in C_c(Y_0(1)).
\]
Since $\mu$ and $\mu_f$ are finite, they extend to the space of
bounded continuous functions on $Y_0(1)$, where we shall denote
also by $\mu$ and $\mu_f$ their extensions.  In particular,
$\mu(1)$ denotes the volume of $Y_0(1)$ and $\mu_f(1)$ the
Petersson norm of $f$.

As is customary, we let $\eps > 0$ denote a sufficiently small
positive number whose precise value may change from line to
line.  We use the asymptotic notation $f(x,y,z) \ll_{x,y}
g(x,y,z)$ to indicate that there exists a positive real
$C(x,y)$, possibly depending upon $x$ and $y$ but not
upon $z$, such that $| f(x,y,z)| \leq C(x,y) |g(x,y,z)|$ for all
$x,y$, and $z$ under consideration.  We write $f(x,y,z) =
O_{x,y}(g(x,y,z))$ synonymously for $f(x,y,z) \ll_{x,y}
g(x,y,z)$ and write $f(x,y,z) \asymp_{x,y} g(x,y,z)$ synonymously for
$f(x,y,z) \ll_{x,y} g(x,y,z) \ll_{x,y} f(x,y,z)$.

\subsection{Weyl's criterion}
The following standard lemma provides essential motivation for
what follows.
\begin{lemma}\label{prop:3}
  Suppose that
  for each fixed Maass eigencuspform
  or incomplete Eisenstein series $\phi$,
  we have
  \[
  \frac{\mu_f(\phi)}{\mu_f(1)}
  \rightarrow \frac{\mu(\phi)}{\mu(1)}
  \quad
  \text{ as } q k \rightarrow \infty
  \]
  for $q$ squarefree and $f$ a holomorphic newform of weight $k$
  and level $q$;
  the convergence need not be uniform in $\phi$.
  Then Theorem \ref{thm:1} is true.
\end{lemma}
\begin{proof}
  The family of
  probability measures $\phi \mapsto \mu_f(\phi)/\mu_f(1)$
  obtained as $f$ varies
  is equicontinuous for the supremum
  norm on $C_c(Y_0(1))$,
  since
  $\left\lvert \mu_f(\phi_1)/\mu_f(1)
    - \mu_f(\phi_2)/\mu_f(1)
  \right\rvert \leq \sup \left\lvert \phi_1 - \phi_2 \right\rvert$
  for any bounded functions $\phi_1, \phi_2$
  on $Y_0(1)$.
  Thus Theorem \ref{thm:1} follows
  if we can show that $\mu_f(\phi)/\mu_f(1) \rightarrow
  \mu(\phi)/\mu(1)$ as $q \rightarrow \infty$ for a set of
  bounded functions $\phi$ the uniform closure of whose span contains
  $C_c(Y_0(1))$; such a set is furnished
  \cite{MR1942691}
  by the Maass eigencuspforms and incomplete Eisenstein series
  as defined in
  \S \ref{sec:modular-forms-their}.
\end{proof}

\subsection{Acknowledgements}
\thanks{We thank Dinakar Ramakrishnan for suggesting this
  problem and for his very helpful feedback and comments on
  earlier drafts of this paper.  We thank Abhishek Saha for his
  careful reading of, and useful comments on, an earlier draft.
  The problem that we address was raised explicitly by
  K. Soundararajan at the 2010 Arizona Winter School;
  we thank him as well as Roman Holowinsky,
  Henryk Iwaniec,
  Philippe Michel, and Peter Sarnak for their encouragement.
  We thank the referees for their numerous
  helpful suggestions and corrections.
  This work represents part of the author's doctoral
  disseration written at the California
  Institute of Technology.
}

\section{Background on automorphic
  forms}\label{sec:modular-forms-their}
We collect here some standard properties
of classical automorphic forms.
We refer to
Serre \cite{MR0344216}, Shimura \cite{MR0314766},
Iwaniec \cite{Iw97,MR1942691}
and Atkin-Lehner \cite{MR0268123}
for complete definitions and proofs.

\subsection{Holomorphic
  newforms}\label{sec:holomorphic-newforms}
Let $k$ be a positive even integer,
and let $\alpha$ be an element of $\GL(2,\mathbb{R})$ with
positive determinant; the element $\alpha$ acts on
$\mathbb{H}$ by fractional linear transformations
in the usual way.
Given a function $f :  \mathbb{H} \rightarrow \mathbb{C}$,
we denote by $f|_k \alpha$ the function
$z \mapsto \det(\alpha)^{k/2} j(\alpha,z)^{-k} f(\alpha
z)$,
where $j \left( \left(
    \begin{smallmatrix}
      a&b\\
      c&d
    \end{smallmatrix}
  \right), z  \right) = c z + d$.

A \emph{holomorphic cusp form} on $\Gamma_0(q)$ of weight $k$ is
a holomorphic function $f : \mathbb{H} \rightarrow \mathbb{C} $
that satisfies $f|_k \gamma = f$ for all $\gamma \in
\Gamma_0(q)$ and vanishes at the cusps of $\Gamma_0(q)$.  A
\emph{holomorphic newform} is a cusp form that is an eigenform
of the algebra of Hecke operators and orthogonal with respect to
the Petersson inner product to the oldforms.\footnote{The terms
  we leave undefined are standard and their precise definitions,
  which may be found in the references mentioned above, are not
  necessary for our purposes.}
We say
that a holomorphic newform $f$ is a \emph{normalized
  holomorphic newform} if moreover $\lambda_f(1) = 1$ in the
Fourier expansion
\begin{equation}\label{eq:f-fourier}
  y^{k/2} f(z) = \sum_{n \in \mathbb{N}}
  \frac{\lambda_f(n)}{\sqrt{n}}
  \kappa_f(n y) e(n x),
\end{equation}
where $\kappa_f(y) = y^{k/2} e^{- 2 \pi y}$ and $e(x) = e^{2 \pi
  i x}$; in that case the Fourier coefficients $\lambda_f(n)$
are real, multiplicative, and satisfy
\cite{deligne-l-adic,deligne-weil-1} the Deligne bound
$|\lambda_f(n)| \leq \tau(n)$, where $\tau (n)$ denotes the
number of positive divisors of $n$.
If $\gamma \in \Gamma_0(q)$
and $z ' = \gamma z = x' + i y ' $,
then $y'^{k/2} f(z') = (j(\gamma,z)/|j(\gamma,z)|)^k y^{k/2}
f(z)$,
so that in particular $z \mapsto y^k |f(z)|^2$ is
$\Gamma_0(q)$-invariant
and our definition of $\mu_f$ given in Section
\ref{sec:notation-conventions}
makes sense.

To a newform $f$ one attaches the finite part
of the adjoint $L$-function
$L(\ad f, s) = \prod_p L_p(\ad f,s)$
and its completion $\Lambda(\ad f, s)
= L_\infty(\ad f, s) L(\ad f,s)
= \prod_v L_v(\ad f, s)$,
where $p$ traverses the set of primes
and $v$ the set of places of $\mathbb{Q}$;
the local factors $L_v(\ad f, s)$ are
as in
\cite[\S 3.1.1]{watson-2008}.
The Rankin-Selberg method \cite{Ra39,Se40}
and a standard calculation
\cite[\S 3.2.1]{watson-2008} show that
\begin{equation}\label{eq:2}
  \mu_f(1) :=
  \int _{\Gamma_0(q) \backslash \mathbb{H} }
  |f|^2(z) y^k \, \frac{d x \, d y}{y^2}
  =
  q
  \frac{\Gamma(k-1)}{(4 \pi)^{k-1}}
  \frac{k - 1}{2 \pi ^2 }
  L(\ad f,1).
\end{equation}
As in the analogous weight aspect
\cite[p.7]{MR2680499}, the work of
Gelbart-Jacquet \cite{MR533066} (following Shimura \cite{Sh75})
and the theorem of Hoffstein-Lockhart \cite[Theorem 0.1]{HL94}
(with appendix by Goldfeld-Hoffstein-Lieman) imply that
\begin{equation}\label{eq:hlghl}
  L(\ad f,1)^{-1} \ll \log(q k).
\end{equation}

Let $\sigma$ traverse a set of representatives
for the double coset space $\Gamma_\infty \backslash \Gamma /
\Gamma_0(q)$.
Then the points $\mathfrak{a}_\sigma  := \sigma^{-1} \infty
\in \mathbb{P}^1(\mathbb{Q})$
traverse a set of inequivalent cusps of $\Gamma_0(q)$.
The integer
$d_\sigma := [\Gamma_\infty : \Gamma_\infty \cap \sigma
\Gamma_0(q) \sigma^{-1}]$ is the width of the cusp
$\mathfrak{a}_\sigma$,
while
\[
w_\sigma := \sigma^{-1} \begin{pmatrix}
  d_\sigma  &  \\
   & 1
\end{pmatrix}
\]
is the scaling matrix for $\mathfrak{a}_\sigma$,
which means that $z \mapsto z_\sigma := w_\sigma z$
is a proper isometry of $\mathbb{H}$ under which
$z_\sigma \mapsto z_\sigma + 1$
corresponds
to the action on $z$ by
a generator for the $\Gamma_0(q)$-stabilizer of
$\mathfrak{a}_\sigma$.

If the bottom row of $\sigma^{-1}$
is $(c,d)$, then $d_\sigma = q / (q,c^2)$;
moreover, as $\sigma$ varies,
the multiset of widths $\{d_\sigma\}$ is
the set $\{d : d|q\}$ of positive divisors of $q$ \cite[\S
2.4]{MR1942691}.
In particular, $c$ and $d_\sigma$ are coprime,
so we may and shall assume (after multiplying
$\sigma$ on the left by an element of $\Gamma_\infty$
if necessary)
that $d_\sigma$ divides $d$.
Since $q$ is squarefree,
the numbers $d_\sigma$ and $q/d_\sigma$ are coprime,
so that $w_\sigma$ is an Atkin-Lehner operator
``$W_Q$''
in the sense of \cite[p.138]{MR0268123}.
Thus by applying \cite[Thm 3]{MR0268123} to the newform
$f$, we obtain
\begin{equation}\label{eq:atkin-lehner}
  f |_k w_\sigma = \pm f.
\end{equation}
Since $f$ is $\Gamma_0(q)$-invariant,
the property \eqref{eq:atkin-lehner} does not depend
upon the choice of coset representative $\sigma$.

\subsection{Maass eigencuspforms}
A \emph{Maass cusp form} (of level $1$) is a $\Gamma$-invariant
eigenfunction of the hyperbolic Laplacian $\Delta := y^{-2}
(\partial_x^2 + \partial_y^2)$ on $\mathbb{H}$ that decays
rapidly at the cusp of $\Gamma$.  By Selberg's ``$\lambda_1 \geq
1/4$'' theorem \cite{MR0182610} there exists a real number $r
\in \mathbb{R}$ such that $(\Delta + 1/4 + r^2) \phi =
0$;
our arguments use only that $r \in \mathbb{R} \cup
i(-1/2,1/2)$, and so apply verbatim in contexts where
``$\lambda_1 \geq 1/4$'' is not known.

A \emph{Maass
  eigencuspform} is a Maass cusp form that is an eigenfunction
of the (non-archimedean) Hecke operators and the involution
$T_{-1} : \phi \mapsto [z \mapsto \phi(-\bar{z})]$,
which commute one another as well as with $\Delta$.
A Maass eigencuspform $\phi$ has
a Fourier expansion
\begin{equation}\label{eq:phi-fourier}
  \phi (z) = \sum _{n \in \mathbb{Z} _{\neq 0}}
  \frac{\lambda_\phi(n)}{\sqrt{|n|}}
  \kappa_{i r}(n y) e( n x)
\end{equation}
where $\kappa_{i r}(y) = 2 |y|^{1/2} K_{i r}(2 \pi |y|)
\sgn(y)^{\frac{1+\delta}{2}}$ with
$K_{ir}$ the standard $K$-Bessel function, $\sgn(y) = 1$ or $-1$
according
as $y$ is positive or negative,
and $\delta \in \{\pm 1\}$ the $T_{-1}$-eigenvalue of $\phi$.  We have
$|\kappa_s(y)| \leq 1$ for all $s \in i \mathbb{R} \cup
(-1/2,1/2)$ and all $y \in \mathbb{R}_+^*$.  
A \emph{normalized Maass
  eigencuspform} further satisfies $\lambda_\phi(1) = 1$; in
that case the coefficients $\lambda_\phi(n)$ are real,
multiplicative, and
satisfy,
for each
$x \geq 1$, the Rankin-Selberg bound \cite[Theorem
3.2]{MR1942691}
\begin{equation}\label{eq:rs-bound}
  \sum_{n \leq x} \lvert \lambda_\phi(n) \rvert^2
  \ll_\phi  x.
\end{equation}

Because $f(-\bar{z}) = \overline{f(z)}$ for any normalized holomorphic
newform $f$, we have $\mu_f(\phi) = 0$ whenever
$T_{-1} \phi = \delta \phi$ with $\delta = -1$.  Thus we shall assume throughout this
paper
that $\delta = 1$, i.e., that $\phi$ is an \emph{even} Maass form.

\subsection{Eisenstein series}
Let $s \in \mathbb{C}$ and $z \in \mathbb{H}$.
The \emph{real-analytic Eisenstein series}
$E(s,z) = \sum_{\Gamma_\infty \backslash \Gamma}
  \Im(\gamma z)^s$
converges normally for $\Re(s) > 1$ and continues
meromorphically to the half-plane $\Re(s) \geq 1/2$ where
the map $s \mapsto E(s,z)$ is
holomorphic with the exception of a unique simple pole at $s =
1$ of constant residue $\res_{s=1} E(s,z) = \mu(1)^{-1}$.
The Eisenstein
series satisfies the invariance
$E(s,\gamma z)
= E(s,z)$ for all $\gamma \in \Gamma$
and
admits the Fourier expansion
\begin{equation}\label{eq:24}
  E(s,z)
  = y^s + M(s)  y^{1-s} 
  + \frac{1}{\xi(2s )}\sum_{n \in \mathbb{Z}_{\neq 0}}
  \frac{\lambda_{s-1/2}(n)}{\sqrt{|n|}}
  \kappa_{s-1/2}(n y) e(n x),
\end{equation}
where $\lambda_s(n) = \sum_{a b = n} (a/b)^s$, $\kappa_s(y) = 2
|y|^{1/2} K_s(2 \pi |y|)$, $M(s) = \xi(2s-1)/\xi(2 s)$, $\xi(s)
= \Gamma_\mathbb{R}(s) \zeta(s)$, $\Gamma_\mathbb{R}(s) =
\pi^{-s/2} \Gamma(s/2)$, and $\zeta(s) = \sum_{n \in \mathbb{N}}
n^{-s}$ (for $\Re(s) > 1$) is the Riemann zeta function.
The identity $|M(s)| = 1$ for $\Re(s) = 1/2$
follows from
(for instance) the functional equation for the zeta function and the prime
number theorem.
When $\Re(s) = 1/2$ we call $E(s,z)$ a
\emph{unitary Eisenstein series}.

\subsection{Incomplete Eisenstein series}
Let $\Psi \in C_c^\infty(\mathbb{R}_+^*)$ be a nonnegative-valued
test function with Mellin transform
$\Psi ^\wedge(s) = \int _0^\infty \Psi(y) y^{-s-1} \, d y$.
Repeated partial integration shows that
$|\Psi^\wedge(s)| \ll_{\Psi,A} (1 + |s|)^A$ for each positive
integer $A$,
uniformly for $s$ in vertical strips.
The Mellin inversion formula asserts that
$\Psi(y) = \int _{(2)} \Psi^\wedge(s) y^{s} \, \frac{d s}{2 \pi
  i}$,
where $\int_{(\sigma)}$ denotes the integral taken over the vertical
contour
from $\sigma - i \infty$ to $\sigma + i \infty$.
To such $\Psi $ we attach the \emph{incomplete Eisenstein series}
\begin{equation}\label{eq:3}
  E(\Psi,z) = \sum_{\gamma \in \Gamma _\infty \backslash \Gamma}
  \Psi( \Im(\gamma z)).
\end{equation}
The sum has a uniformly bounded finite number
of nonzero terms for $z$ in a fixed compact subset of $\mathbb{H}$.
By Mellin inversion, the rapid decay of
$\Psi^\wedge$ and Cauchy's theorem, we have
\begin{equation}\label{eq:1}
  E(\Psi,z)
  = \int _{(2 )}
  \Psi^\wedge(s) E(s,z) \, \frac{d s}{2 \pi i }
  =
  \frac{\Psi^\wedge(1)}{\vol(\Gamma \backslash \mathbb{H})}
  + \int_{(1/2)} \Psi^\wedge(s) E(s,z) \, \frac{d s}{ 2 \pi i}.
\end{equation}

Let $\phi = E(\Psi,\cdot)$ be an incomplete Eisenstein series.
Note that $\mu(\phi) = \Psi^\wedge(1)$.
By comparing
(\ref{eq:1}) and \eqref{eq:24},
we may
express the Fourier coefficients $\phi_n(y)$ 
in the Fourier series
$\phi(z) = \sum_{n \in \mathbb{Z}} \phi_n(y) e(n x)$ as
\begin{align}\label{eq:9}
  \phi_n(y) &= \int_{(1/2)}
  \frac{\Psi^\wedge(s)}{\xi(2s )} \frac{\lambda_{s-1/2}(n)}{\sqrt{|n|}}
  \kappa_{s-1/2}(n y) \, \frac{d s}{2 \pi i}
  &(n \neq 0), \\ \label{eq:10}
  \phi_0(y) &= \frac{\mu(\phi)}{\mu(1)}
  + \int_{(1/2)}
  \Psi^\wedge(s) \left( y^s + M(s) y^{1-s} \right)
  \, \frac{d s}{2 \pi i} &(n = 0).
\end{align}

\section{Main estimates}\label{sec:main-estimates}
We prove a level aspect analogue of Holowinsky's
main bound \cite[Corollary 3]{MR2680498}.
To formulate our result, 
define for each normalized holomorphic newform $f$
and each real number $x \geq 1$ the quantities
\begin{equation}\label{eq:12}
  M_f(x)
  = \frac{\prod_{p \leq x} (1 + 2 |\lambda_f(p)|/p)}{\log(e x)^2
    L(\ad f,1)},
  \quad
  R_f(x) =
  \frac{x^{-1/2}}{L(\ad f,1)}
  \int_{\mathbb{R}}
  \left\lvert \frac{L(\ad f, \tfrac{1}{2} + i t)}{(1 +
      |t|)^{10}} \right\rvert
  \, d t.
\end{equation}
In \S \ref{sec:proof-theor-refthm:1} we shall refer only to
the definitions (\ref{eq:12}) and the statement of the following
theorem, not its proof.
\begin{theorem}\label{thm:holow}
  Let $f$ be a normalized holomorphic
  newform of weight $k$ and squarefree
  level $q$.
  If $\phi$ is a Maass eigencuspform, then
  \[
  \frac{\mu_f(\phi)}{\mu_f(1)} \ll_{\phi,\eps}
  \log(qk)^{\eps} M_f(qk)^{1/2}.
  \]
  If $\phi$ is an incomplete Eisenstein series, then
  \[
  \frac{\mu_f(\phi)}{\mu_f(1)}
  - \frac{\mu(\phi)}{\mu (1) }
  \ll_{\phi,\eps}
  \log(qk)^\eps M_f(qk)^{1/2} \left( 1 + R_f(qk) \right).
  \]
\end{theorem}

In this section $k$ is a positive even integer, $f$ is a
normalized holomorphic newform of weight $k$ and squarefree
level $q$, and $\phi$ is a Maass eigencuspform or incomplete
Eisenstein series.  In \S \ref{sec:reduct-shift-sums} we reduce
Theorem \ref{thm:holow} to a problem of estimating shifted sums
(see Definition \ref{defn:shifted-sums}).  In \S
\ref{sec:analys-shift-sums} we apply a refinement of
\cite[Theorem 2]{MR2680498} to bound such
shifted sums.  In \S
\ref{sec:sums-shifted-sums} we complete the proof of Theorem
\ref{thm:holow}.

\subsection{Reduction to shifted
  sums}\label{sec:reduct-shift-sums}
Fix once and for all an everywhere nonnegative test
function $h \in C_c^\infty(\mathbb{R}_+^*)$
with Mellin transform
$h^\wedge(s) = \int_0^\infty h(y) y^{-s-1} \, d y$
such that $h^\wedge(1) = \mu(1)$.
In what follows, all implied constants
may depend upon $h$ without mention.
\begin{definition}\label{defn:shifted-sums}
  To the parameters $s \in \mathbb{C}$, $l \in \mathbb{Z}_{\neq 0}$
  and $x \geq 1$ we associate
  the \emph{shifted sums}
  \[
  S_s(l,x) = \sum
  _{\substack{
      n \in \mathbb{N} \\
      m := n + l \in \mathbb{N} 
    }}
  \frac{\lambda_f(m)}{\sqrt{m}}
  \frac{\lambda_f(n)}{\sqrt{n}}
  I_s(l,n,x),
  \]
  where $I_s(l,n,x)$ is an integral depending
  upon our fixed test function $h$:
  \[
  I_s(l,n,x)
  = \int _0 ^\infty h(x y) \kappa_s(l y)
  \kappa_f(m y) \kappa_f(n y) y^{-1} \, \frac{d y}{y},
  \quad m := n + l.
  \]
\end{definition}
Our aim in this section is to reduce Theorem \ref{thm:holow} to
the problem of bounding such shifted sums.  We shall
subsequently refer to the statement below of Proposition
\ref{prop:reduce-to-sums} but not the details of its proof.
\begin{proposition}\label{prop:reduce-to-sums}
  Let $Y \geq 1$.
  If $\phi$ is a Maass eigencuspform of eigenvalue
  $1/4 + r^2$, then
  \[
  \frac{\mu_f(\phi)}{\mu_f(1)}
  = \frac{1}{Y \mu_f(1)}
  \sum_{\substack{
      l \in \mathbb{Z}_{\neq 0} \\
      |l| < Y^{1+\eps}
    }
  }
  \frac{\lambda_\phi(l)}{ \sqrt{|l|}}
  \sum_{d|q} S_{i r}(d l, d Y)
  + O_{\phi,\eps}(Y^{-1/2}).
  \]
  If $\phi = E(\Psi,\cdot)$ is an incomplete Eisenstein series,
  then
  \[
  \begin{split}
    \frac{\mu_f(\phi)}{\mu_f(1)}
    - \frac{\mu(\phi)}{\mu(1)}
    & =
    \frac{1}{Y \mu_f(1)}
    \int _\mathbb{R} 
    \frac{\Psi^\wedge(\tfrac{1}{2}+it)}{\xi(1+2 i t)}
    \left(
      \sum_{\substack{
          l \in \mathbb{Z}_{\neq 0} \\
          |l| < Y^{1+\eps}
        }
      }
      \frac{\lambda_{i t}(l)}{\sqrt{|l|}}
      \sum_{d|q} S_{it}(d l, d Y)
    \right)
    \frac{d t}{2 \pi } \\
    & \quad + O_{\phi,\eps} \left( \frac{1 + R_f(q k)}{Y^{1/2}}
    \right).
  \end{split}
  \]
\end{proposition}

Our proof follows a sequence of lemmas.
Let $k, f, q, Y, \phi, h$ be as above
and let $h_Y$ be the function
$y \mapsto h(Y y)$.
To $h_Y$ we attach the incomplete Eisenstein series
$E(h_Y,z)$ by the usual recipe (\ref{eq:3}).
\begin{lemma}\label{lem:expand-int-e-of-h}
  We have the following approximate formula
  for the quantity $\mu_f(\phi)$:
  \begin{align*}
    Y \mu_f(\phi) &= \sum_{d|q} \int _{y = 0}^\infty h_Y(d y)
    \int_{x=0}^1 \phi( d z) |f|^2(z) y^k \, \frac{d x \, d y }{
      y^2} + O_{\phi}(Y^{1/2} \mu_f(1)).
  \end{align*}
\end{lemma}
\begin{proof}
  By Mellin inversion and Cauchy's theorem as in (\ref{eq:1}),
  we have
  \begin{equation*}
    Y
    \mu_f(\phi)
    = \mu_f(E(h_Y,\cdot) \phi)
    - \int_{(1/2)}
    h^\wedge(s) Y^s \mu_f(E(s,\cdot)\phi) \, \frac{d s}{2 \pi i}.
  \end{equation*}
  The argument of \cite[Proof of Lemma
  3.1a]{MR2680498} shows without modification that
  \begin{equation}\label{eq:29}
    \int_{(1/2)}
    h^\wedge(s) Y^s \mu_f(E(s,\cdot)\phi) \, \frac{d s}{2 \pi
      i}  \ll_{\phi} Y^{1/2} \mu_f(1);
  \end{equation}
  since the proof is short, we sketch it here.  By the Fourier
  expansion for $E(s,z)$ and the rapid decay of $\phi(z)$ as $y
  \rightarrow \infty$, we have $E(s,z) \phi(z) \ll_\phi
  |s|^{O(1)}$ for $\Re(s) = 1/2$ and $z$ in the Siegel domain
  $\{z : x \in [0,1], y > 1/2 \}$ for $\Gamma \backslash
  \mathbb{H}$. By the rapid decay of $h^\wedge$ we have
  $h^\wedge(s) Y^s E(s,z) \phi(z) \ll_\phi Y^{1/2} |s|^{-2}$
  for $s,z$ as above; the estimate \eqref{eq:29}
  follows by integrating in $z$ against $\mu_f$
  and then integrating in $s$.

  Having established that $Y \mu_f(\phi) = \mu_f(E(h_Y,\cdot)
  \phi) + O_\phi(Y^{1/2} \mu_f(1))$, it remains now only to
  evaluate $\mu_f(E(h_Y,\cdot)\phi)$.
  Let $\Gamma_\infty \backslash
  \Gamma / \Gamma_0(q) = \{\sigma \}$ be a set of double-coset
  representatives as in \S \ref{sec:holomorphic-newforms},
  and set
  \[
  d_\sigma = [\Gamma_\infty : \Gamma_\infty \cap
  \sigma \Gamma_0(q) \sigma^{-1}]
  .
  \]
  By decomposing the transitive right $\Gamma$-set
  $\Gamma_\infty \backslash \Gamma $ into $\Gamma_0(q)$-orbits
  \[
  \Gamma _\infty \backslash \Gamma
  = \sqcup \Gamma_\infty \backslash \Gamma_\infty \sigma
  \Gamma_0(q)
  = \sqcup  \sigma (
  \sigma^{-1} \Gamma_\infty \sigma \cap \Gamma_0(q) \backslash
  \Gamma_0(q) ),
  \]
  we obtain
  \[
  E(h_Y,z) = \mathop{\sum \sum}_{
    \substack{
      \sigma \in \Gamma_\infty \backslash \Gamma /
      \Gamma_0(q) \\
      \gamma \in \sigma^{-1} \Gamma_\infty
      \sigma \cap \Gamma_0(q) \backslash \Gamma_0(q)
    }
  }
  h_Y(\Im (\sigma \gamma z)).
  \]
  By invoking
  the $\Gamma_0(q)$-invariance of
  $z \mapsto \phi(z) |f|^2(z) y^k \, \frac{d x \, d y}{y^2}$
  and unfolding the sum over $\gamma \in \sigma^{-1}
  \Gamma_\infty \sigma \cap \Gamma_0(q) \backslash \Gamma_0(q)$
  with the integral over $z \in \Gamma_0(q) \backslash
  \mathbb{H}$,
  we get
  \[
  \mu_f(E(h_Y,\cdot)\phi)
  =
  \sum_{\sigma \in \Gamma_\infty \backslash \Gamma /
    \Gamma_0(q)}
  \int_{\sigma^{-1} \Gamma_\infty  \sigma \cap \Gamma_0(q)
    \backslash \mathbb{H}}
  h_Y(\Im (\sigma z)) \phi(z) |f|^2(z) y^k \, \frac{d x \, d
    y}{y^2}.
  \]
  The change of variables
  $z \mapsto \sigma^{-1} z$
  transforms the integral above into
  \[
  \int_{ \Gamma_\infty \cap \sigma \Gamma_0(q) \sigma^{-1}
    \backslash \mathbb{H}}
  h_Y( y) \phi( z)
  |f|^2(\sigma^{-1} z)
  \Im(\sigma^{-1} z)^k
  \, \frac{d x \, d
    y}{y^2}.
  \]
  Integrating over a fundamental domain for $\Gamma_\infty \cap
  \sigma \Gamma_0(q) \sigma^{-1} = \{\pm \left(
    \begin{smallmatrix}
      1& d_\sigma n\\
      &1
    \end{smallmatrix}
  \right) : n \in \mathbb{Z} \}$ acting on
  $\mathbb{H}$, we get
  \[
  \int_{ y=0}^\infty
  h_Y( y)
  \int_{x=0}^{d_\sigma}
  \phi( z)
  |f|^2(\sigma^{-1} z)
  \Im(\sigma^{-1} z)^k
  \, \frac{d x \, d
    y}{y^2}.
  \]
  Applying now the change of variables
  $z \mapsto d_\sigma z$ gives
  \[
  \int_{ y=0}^\infty 
  h_Y(d_\sigma y)
  \int_{x=0}^1
  \phi(d_\sigma z)
  \left\lvert f |_k \sigma^{-1} \left(
      \begin{smallmatrix}
        d_\sigma &\\
        &1
      \end{smallmatrix}
    \right) \right\rvert^2 (z)
  y^k
  \, \frac{d x \, d
    y}{y^2}.
  \]
  Since $f |_k \sigma^{-1} \left(
    \begin{smallmatrix}
      d_\sigma &\\
      &1
    \end{smallmatrix} \right) = \pm f$ by the consequence
  (\ref{eq:atkin-lehner}) of Atkin-Lehner theory
  (using here that $q$ is squarefree), we find that
  \[
  \mu_f(E(h_Y,\cdot)\phi)
  = 
  \sum_{\sigma \in \Gamma_\infty \backslash \Gamma /
    \Gamma_0(q)}
  \int_{ y=0}^\infty 
  h_Y(d_\sigma y)
  \int_{x=0}^1
  \phi(d_\sigma z)
  |f|^2(z)
  y^k
  \, \frac{d x \, d
    y}{y^2}.
  \]
  Since $\{d _\sigma \} = \{ d : d | q\}$, we obtain the
  claimed formula.
\end{proof}

In the expression for $Y \mu_f(\phi)$
given by Lemma \ref{lem:expand-int-e-of-h},
we expand $\phi$ in a Fourier series
$\phi(z) = \sum_{l \in \mathbb{Z}} \phi_l(y) e(l x)$
and consider
separately the contributions
from $l$ in various ranges; specifically, we set
\begin{align*}
  \mathcal{S}_0 &= 
  \sum_{d|q}
  \int_{y=0}^\infty h_Y(d y)
  \int_{x=0}^1 \phi_0(d y) |f|^2(z) y^k
  \, \frac{d x \, d y}{y^2}, \\
  \mathcal{S}_{(0,Y^{1+\eps})} &= 
  \sum_{d|q}
  \int_{y=0}^\infty h_Y(d y)
  \int_{x=0}^1
  \sum_{
    0 < |l| < Y^{1+\eps}
  }
  \phi_l(d y)
  |f|^2(z) y^k
  \, \frac{d x \, d y}{y^2}, \\
  \mathcal{S}_{\geq Y^{1+\eps}} &= 
  \sum_{d|q}
  \int_{y=0}^\infty h_Y(d y)
  \int_{x=0}^1
  \sum_{|l| \geq Y^{1+\eps}}
  \phi_l(d y)
  |f|^2(z) y^k
  \, \frac{d x \, d y}{y^2},
\end{align*}
so that
\begin{equation}\label{eq:4}
  \sum_{d|q}
  \int_{y=0}^\infty h_Y(d y)
  \int_{x=0}^1 \phi(d z) |f|^2(z) y^k
  \, \frac{d x \, d y}{y^2}
  =
  \mathcal{S}_0
  +
  \mathcal{S}_{(0,Y^{1+\eps})}
  +
  \mathcal{S}_{\geq Y^{1+\eps}}.
\end{equation}
We treat these contributions in Lemmas \ref{lem:2a},
\ref{lem:open-main-term} and \ref{lem:2}, respectively; in
doing so we shall repeatedly use the following technical
result.
\begin{lemma}\label{lem:3}
  The quantity $\mu_f(E(h_Y,\cdot))$
  satisfies the
  formulas and estimates
  \begin{eqnarray*}
    \mu_f(E(h_Y,\cdot))
    &=&
    \sum_{d|q}
    \int _{y = 0}^\infty h_Y(d y)
    \int_{x=0}^1 |f|^2(z) y^k \, \frac{d x \, d y }{ y^2} \\
    &=&
    Y \mu_f(1)
    \left( 1 +
      E_f(q Y)
    \right)
    \\
    &=& Y \mu_f(1) \left( 1
      + O \left( Y^{-1/2} R_f(q k) \right)
    \right),
  \end{eqnarray*}
  where
  \[
  E_f(x)
  :=
  \frac{2 \pi^2}{x}
  \int_{(1/2)}
  h^\wedge(s) 
  \left(
    \frac{x}{4 \pi}
  \right)^s
  \frac{\Gamma(s+k-1)}{\Gamma(k)}
  \frac{\zeta(s)}{\zeta(2 s)} 
  \frac{L(\ad f, s)}{L(\ad f, 1)}
  \,
  \frac{
    d s}{2 \pi i}.
  \]
  Moreover, $\mu_f(E(h_Y,\cdot))
  \ll Y \mu_f(1)$.
\end{lemma}
\begin{proof}
  The first equality
  follows from the same argument as in the proof
  of Lemma \ref{lem:expand-int-e-of-h},
  the second from the Mellin formula and the
  unfolding method by a direct computation, the third from the
  bounds $|\Gamma(k - 1/2 + i t)| \leq \Gamma(k -
  1/2)| \ll k^{-1/2} \Gamma(k)$, $\zeta(1/2+it) \ll (1 + |t|)^{1/4}$
  and $|\zeta(1 + 2 i t)| \gg 1 / \log(1+|t|)$ as in
  \cite[p.7]{MR2680497}.
  Finally, because the quantity $\mu_f(E(h_Y,\cdot))$
  is majorized by the integral of the $\Gamma$-invariant measure
  $\mu_f$ over the region on which
  the function $\Gamma_\infty \backslash \mathbb{H} \ni z \mapsto
  h_Y(y)$ does not vanish and because that region
  intersects $\ll Y$ fundamental domains for $\Gamma
  \backslash \mathbb{H}$ \cite[Lemma 2.10]{MR1942691},
  we have $\mu_f(E(h_Y,\cdot)) \ll Y \mu_f(1)$.
\end{proof}

\begin{lemma}[The main term $\mathcal{S}_0$]\label{lem:2a}
  If $\phi$ is a Maass eigencuspform, then 
  $\phi_0(y) = 0$ and $\mathcal{S}_0 =0$.  If $\phi$ is an
  incomplete Eisenstein series, then
  \[
  \mathcal{S}_0
  = Y \mu_f(1)
  \left( \frac{\mu(\phi)}{\mu(1)} + O_\phi \left( \frac{1+R_f(q k)}{Y^{1/2}} \right) \right).
  \]
\end{lemma}
\begin{proof}
  If $\phi$ is a Maass eigencuspform
  then $\phi_0(y) = 0$ holds by definition,
  hence $\mathcal{S}_0 = 0$.  Suppose
  otherwise that $\phi$ is an incomplete Eisenstein series.  It
  follows from (\ref{eq:10}) that for every $y \in \mathbb{R}_+^*$
  such that $h_Y(y) \neq 0$, we have $\phi_0(y) = \mu(\phi)/
  \mu(1) + O_\phi(Y^{-1/2})$.  Thus two applications of Lemma
  \ref{lem:3} show that
  \begin{eqnarray*}
    \mathcal{S}_0
    &=& \mu_f(E(h_Y,\cdot)) \left( \frac{\mu(\phi)}{\mu(1)} +
      O_\phi (Y^{-1/2}) \right) \\
    &=& Y \mu_f(1)
    \left( 1 + O \left( \frac{R_f(q k)}{Y^{1/2}} \right) \right)
    \left( \frac{\mu(\phi)}{\mu(1)}
      + O_\phi (Y^{-1/2}) \right) \\
    &=& Y \mu_f(1)
    \left( \frac{\mu(\phi)}{\mu(1)} + O_\phi \left(  \frac{1 + R_f(q k)}{Y^{1/2}} \right) \right).
  \end{eqnarray*}
\end{proof}

\begin{lemma}[The essential
  error term $\mathcal{S}_{(0,Y^{1+\eps})}$]\label{lem:open-main-term}
  If $\phi$ is a Maass eigencuspform, then
  \[
  \mathcal{S}_{(0, Y^{1+\eps})}
  =
  \sum_{
    0 < |l| < Y^{1+\eps}
  }
  \frac{\lambda_\phi(l)}{\sqrt{|l|}}
  \sum_{d|q}
  S_{i r}( d l, d Y).
  \]
  If $\phi$ is an incomplete Eisenstein series,
  then
  \[
  \mathcal{S}_{(0, Y^{1+\eps})}
  =
  \int_\mathbb{R} 
  \frac{\Psi^\wedge(\tfrac{1}{2}+it)}{\xi(1+2it)}
  \sum_{
    0 < |l| < Y^{1+\eps}
  }
  \frac{\lambda_{i t}(l)}{\sqrt{|l|}}
  \sum_{d|q}
  S_{i t}( d l, d Y) \, \frac{d t }{ 2 \pi }.
  \]
\end{lemma}
\begin{proof}
  Follows by integrating the Fourier expansion
  (\ref{eq:f-fourier}) of a newform, the Fourier expansion
  (\ref{eq:phi-fourier}) of a Maass cusp form, and the formula
  (\ref{eq:9}) for the non-constant Fourier coefficients of an
  Eisenstein series.
\end{proof}

\begin{lemma}[The trivial error term $\mathcal{S}_{\geq Y^{1+\eps}}$]\label{lem:2}
  We have $\mathcal{S}_{\geq Y^{1+\eps}}
  \ll_{\phi,\eps} Y^{-10} \mu_f(1)$.
\end{lemma}
\begin{proof}
  Lemma \ref{lem:2} follows from Lemma \ref{lem:3} and the
  following claim: for all $y \in \mathbb{R}_+^*$ such that
  $h_Y(y) \neq 0$, we have $\sum _{ |l| \geq Y^{1+\eps} }
  \left\lvert \phi_l(y) \right\rvert \ll_{\phi,\eps} Y^{-11}$.
  The claim is proved in \cite[\S
  3.2]{MR2680498}, as follows.  When $\phi$ is a cusp form
  of eigenvalue $1/4 + r^2$, so that $\phi_l(y) = y^{-1/2}
  \lambda_\phi(l) \kappa_{i r}(ly)$, the claim follows from
  the exponential decay of $l \mapsto \kappa_{ir}(l y)$ for $l \geq
  Y^{1+\eps}$ and $y \asymp Y^{-1}$ together with the polynomial
  growth of $l \mapsto \lambda_\phi(l)$.  When $\phi$ is an
  incomplete Eisenstein series, the integral formula
  (\ref{eq:9}) and standard bounds for the $K$-Bessel function
  show that for each positive integer $A$, we have $\phi_l(y)
  \ll_{\phi,\eps,A} \tau(l) Y^{A-1/2} |l|^{-A} ( 1 + Y/|l|
  )^\eps$; the claim then follows by summing over $|l| \geq
  Y^{1+\eps}$.
\end{proof}

\begin{proof}[Proof of Proposition \ref{prop:reduce-to-sums}]
  By Lemma
  \ref{lem:expand-int-e-of-h} and
  equation (\ref{eq:4}),
  we have
  \[
  \frac{\mu_f(\phi)}{\mu_f(1)}
  = \frac{1}{Y \mu_f(1)}
  \left(
    \mathcal{S}_0
    +
    \mathcal{S}_{(0,Y^{1+\eps})}
    +
    \mathcal{S}_{\geq Y^{1+\eps}}
  \right)
  + O_{\phi,\eps}(Y^{-1/2}).
  \]
  Proposition \ref{prop:reduce-to-sums}
  follows by combining the results of
  Lemma \ref{lem:2a}, Lemma \ref{lem:2}
  and Lemma \ref{lem:open-main-term}.
\end{proof}

\subsection{Bounds for individual shifted
  sums}\label{sec:analys-shift-sums}
We bound the individual shifted sums appearing in Definition
\ref{defn:shifted-sums};
in subsequent sections we shall
need only
our main result, Corollary \ref{cor:bound-sums}.
We
first recall
a special case of
Holowinsky's bound \cite[Theorem 2]{MR2680498}.
\begin{theorem}[Holowinsky]\label{thm:holow-sums}
  Let $\eps \in (0,1)$.
  Then for $x \geq 1$ and $l \in \mathbb{Z}_{\neq 0}$,
  we have
  \[
  \sum _{\substack{
      n \in \mathbb{N} \\
      m := n + l \in \mathbb{N} \\
      \max(m,n) \leq x
    }
  }
  | \lambda_f(m) \lambda_f(n) |
  \ll_\eps \tau(l)  \frac{
    x
    \prod_{p
      \leq x}
    ( 1+ 2 |\lambda_f(p)|/p )
  }{ \log(e x)^{2-\eps} }.
  \]
\end{theorem}
Unfortunately, Theorem \ref{thm:holow-sums} is insufficient for
our purposes because $\tau(q l)$ can be quite large,
even larger asymptotically than every power of $\log(e q)$, when $q$ has many
small prime factors.  The following
refinement will suffice.
\begin{theorem}\label{thm:holow-sums-2}
  With conditions as in
  the statement of Theorem \ref{thm:holow-sums},
  we have
  \begin{equation}\label{eq:holow-sums-2}
    \sum _{\substack{
        n \in \mathbb{N} \\
        m := n + l \in \mathbb{N} \\
        \max(m,n) \leq x
      }
    }
    | \lambda_f(m) \lambda_f(n) |
    \ll_\eps
    \frac{
      x
      \prod_{p
        \leq x}
      ( 1+ 2 |\lambda_f(p)|/p )
    }{ \log(e x)^{2-\eps} } 
  \end{equation}
  where all implied constants are absolute.
\end{theorem}
\begin{proof}
  In \cite[Thm 3.1]{PDN-HMQUE}, we generalized Holowinsky's
  bound \cite[Thm 2]{MR2680498} to totally real number
  fields $\mathbb{F}$.  Along the way we proved a pair of results \cite[Thm
  4.10]{PDN-HMQUE} and \cite[Thm 7.2]{PDN-HMQUE} either of which
  imply Theorem \ref{thm:holow-sums-2}.
  For completeness,
  we shall give the argument here
  in the special case
  $\mathbb{F} = \mathbb{Q}$,
  which borrows heavily from that of Holowinsky;
  up to \eqref{eq:31} we essentially
  recall his argument, and after that introduce our refinement.
  Let $\lambda(n) = |\lambda_f(n)|$, so that
  \begin{equation}\label{eq:8}
    \text{$\lambda$ is a nonnegative
      multiplicative function
      satisfying $\lambda(n) \leq \tau(n)$.}
  \end{equation}
  We may assume that $1 \leq l \leq x$.
  Fix $\alpha \in (0,1/2)$ (to be chosen sufficiently
  small at the end of the proof)
  and set
  \[
  y = x^{\alpha},
  \quad 
  s = \alpha \log \log (x),
  \quad
  z = x^{1/s}.
  \]
  For $x \gg_\alpha 1$ we have $10 \leq z \leq y \leq x$,
  as we shall henceforth assume.
  For each $n \in \mathbb{N}$,
  write $m = n + l \in \mathbb{N}$.
  Define the \emph{$z$-part} of a positive integer
  to be the greatest
  divisor of that integer supported on primes $p \leq z$.
  There exist unique positive integers $a,b,c$
  such that $\gcd(a,b)=1$
  and $a c$ (resp. $b c$) is the $z$-part of
  $m$ (resp. $n$);
  such triples $a,b,c$ satisfy
  \begin{equation}\label{eq:7}
    p | a b c \Rightarrow p \leq z,
    \quad c | l,
    \quad \text{and } \gcd(a,b)=1.
  \end{equation}
  Write
  $\mathbb{N} = \sqcup_{a,b,c}\mathbb{N}_{abc}$
  for the fibers of
  $n \mapsto (a,b,c)$.
  The assumption \eqref{eq:8}
  implies $\lambda(m) \lambda(n) \leq 4^s \lambda(a c) \lambda(b
  c)$,
  so that
  \begin{eqnarray*}
    \sum_{n \in \mathbb{N} \cap [1,x]}
    \lambda(m) \lambda(n)
    &=&
    \sum_{a,b,c}
    \sum_{n \in \mathbb{N}_{abc} \cap [1,x]}
    \lambda(m) \lambda(n) \\
    &\leq& 4^s
    \sum_{a,b,c}
    \lambda(a c) \lambda(b c)
    \cdot 
    \# (\mathbb{N}_{abc} \cap [1,x]).
  \end{eqnarray*}
  Holowinsky asserts that Rankin's trick implies that the
  contribution to the above from $a,b,c$ for which $|a c| > y$
  or $|b c| > y$ is $\ll_{\alpha,A} x \log(x)^{-A}$ for any $A$;
  we spell out an alternate proof of this assertion in
  \cite[Lemma 7.3]{PDN-HMQUE}.  Now, an integer belongs to
  $\mathbb{N}_{abc}$ only if it satisfies some congruence
  conditions modulo each prime $p \leq z$ (see
  \cite[p.14]{MR2680498}, or \cite[Lemma 7.3]{PDN-HMQUE} for a
  detailed discussion); as in \cite{MR2680498} or
  \cite[Corollary 7.8]{PDN-HMQUE}, an application of the large
  sieve (or Selberg's sieve) shows that if $|a c| \leq y$,
  $|b c| \leq y$ and $x \gg y^2$, then\footnote{This bound is slightly poorer than that obtained by Holowinsky
    because we have been more precise in our calculation of the
    residue classes sieved out by prime divisors of $c^{-1} l$;
    the discrepancy here does not matter in the end.}
  \begin{equation}
    \#(\mathbb{N}_{abc} \cap [1,x])
    \ll \frac{x + (y z)^2}{\log(z)^2}
    \frac{l}{c^2 \phi(a b c^{-1} l)}.
  \end{equation}
  Since $(y z)^2
  \ll_\alpha x$, $\log(z)^2 \asymp_\alpha
  \log \log(x)^{-2}
  \log(x)^2$, $4^s \ll_\eps \log(x)^\eps$ (for $\alpha
  \ll_\eps 1$), and $\phi(a b c^{-1} l) \geq \phi(c^{-1} l)
  \phi(a) \phi(b)$, we see that Theorem \ref{thm:holow-sums-2}
  follows from the bound\footnote{It is here
    that Holowinsky gives up the factor $\tau(l)$.}
  \begin{equation}\label{eq:31}
    \sum_{
      \substack{
        c | l \\
        p | c \Rightarrow p \leq z         
      }
    }
    \frac{1}{c}
    \frac{l/c}{\phi(l/c)}
    \mathop{
      \sum _{|a c| \leq y}
      \sum _{|b c| \leq y}
    }_{ p | a b \Rightarrow p \leq z}
    \frac{\lambda(a c) \lambda(b c)}{\phi(a) \phi(b)}
    \ll \log \log(x)^{O(1)}
    \prod_{p \leq z}
    \left( 1 + \frac{2 \lambda(p)}{p} \right),
  \end{equation}
  which we now establish.
  Note first that
  \begin{equation}
    \mathop{
      \sum _{|a c| \leq y}
      \sum _{|b c| \leq y}
    }_{ p | a b \Rightarrow p \leq z}
    \frac{\lambda(a c) \lambda(b c)}{\phi(a) \phi(b)}
    \leq \left( \prod_{p \leq z} \sum_{k \geq 0}
      \frac{\lambda(p^{k+v_p(c)})}{\phi(p^k)} \right)^2.
  \end{equation}
  Using that $\lambda(p^k) \leq k+1$ and $p \geq 2$
  and summing some geometric series
  as in \cite[Lemma 7.4]{PDN-HMQUE}
  gives
  \begin{equation*}
    \sum_{k \geq 0} \frac{\lambda(p^{k+\nu})}{\phi(p^k)}
    \leq
    \nu + 1 + \sum_{k \geq 1} \frac{\nu + k + 1}{p^{k-1}(p-1)}
    \leq 3 \nu + 3
  \end{equation*}
  for each $\nu \geq 1$,
  while for $\nu = 0$
  \begin{eqnarray*}
    \sum_{k \geq 0} \frac{\lambda(p^k)}{\phi(p^k)}
    &=&
    \left( 1 + \frac{\lambda(p)}{p} \right)
    \left( 1 +
      \frac{
        \lambda(p) \left(\frac{1}{\phi(p)} - \frac{1}{p} \right)
        + \sum_{k \geq 2} \frac{\lambda(p^k)}{\phi(p^k)}
      }
      {
        1 + \frac{\lambda (p)}{p}
      } \right) \\
    &\leq&
    \left( 1 + \frac{\lambda(p)}{p} \right)
    \left( 1 + \frac{20}{p}\right).
  \end{eqnarray*}
  Thus the LHS of \eqref{eq:31}
  is bounded by $\zeta(2)^{40} \psi(l) \prod_{p \leq z} (1 + \lambda(p)
  p^{-1})^2$,
  where $\psi$ is the multiplicative function
  \begin{equation}
    \psi(l) = \sum_{c | l}
    \frac{1}{c}
    \frac{l/c}{\phi(l/c)}
    \prod_{p^\nu || c} (3 \nu + 3)^2.
  \end{equation}
  By direct calculation
  and the inequality $p \geq 2$,
  we have
  \[
  \psi(p^a) = \frac{1}{1 - p ^{-1} }
  + \frac{9}{p ^a } \left(
    (a + 1) ^2 + \frac{1}{1 - p ^{-1} }
    \sum_{i=1}^{a-1}
    \frac{(i+1)^2}{p^i}
  \right)
  \leq 1 + Cp^{-1}
  \]
  for some constant $C \leq 10^6$,
  so that
  $\psi(l) \leq \prod_{p|l} (1 + Cp^{-1})
  \ll \log \log(x)^{C}$
  for $1 \leq l \leq x$.
  This estimate for $\psi(l)$ establishes the claimed bound \eqref{eq:31}.
\end{proof}

\begin{remark}\label{rmk:bounds-indiv-shift}
  A bound of the form \eqref{eq:holow-sums-2} but with an
  \emph{unspecified} dependence on the parameter $l$ may be
  derived from the work of Nair \cite{MR1197420}.  We have
  attempted to quantify this dependence by working through the
  details of Nair's arguments, and have shown that
  they imply
  \begin{equation}
    \sum _{\substack{
        n \in \mathbb{N} \\
        m := n + l \in \mathbb{N} \\
        \max(m,n) \leq x
      }
    }
    | \lambda_f(m) \lambda_f(n) |
    \ll_\eps
    \tau_m(l)
    \frac{
      x
      \prod_{p
        \leq x}
      ( 1+ 2 |\lambda_f(p)|/p )
    }{ \log(e x)^{2-\eps} } 
  \end{equation}
  for some $m \geq 2$ (probably $m = 2$) and all $0 \neq |l|
  \leq x^{1/16-\eps}$; in deducing this we have
  used
  the Ramanujan bound $|\lambda_f(p)| \leq 2$.
  This strength and uniformity falls far short
  of what is needed in treating the level aspect of QUE.
  
  A mild strengthening of \eqref{eq:holow-sums-2} subject to the
  additional constraint $4 l^2 \leq x$ appears in the recent
  book of Iwaniec-Friendlander \cite[Thm 15.6]{MR2647984}, which
  was released after we completed the work of this paper.  The
  condition $4 l^2 \leq x$ makes their result inapplicable in
  our treatment of the level aspect of QUE, where $l$ can be
  nearly as large as $x$.  However, it seems to us that one can
  remove this condition by a suitable modification of their
  arguments.
\end{remark}

Recall from Definition \ref{defn:shifted-sums} that the sums
$S_s(l,x)$ involve a certain integral
$I_s(l,n,x)$.
\begin{lemma}\label{lem:bound-integral}
  For each positive integer $A$,
  the integral $I_s(l,n,x)$ satisfies the upper bound
  \[
  I_s(l,n,x) \ll_{A}
  \frac{\Gamma(k-1)}{(4 \pi)^{k-1}}
  \sqrt{m n } \cdot \max \left( 1, \frac{\max(m,n)}{x k}
  \right)^{-A}
  \]
  uniformly for $s \in i \mathbb{R} \cup (-1/2,1/2)$,
  $n \in \mathbb{N}$, $l \in
  \mathbb{Z}_{\neq 0}$, and $x \geq 1$.  Here $m := n + l$, as
  usual.
\end{lemma}
\begin{proof}
  Let $s, l, m, n$ be as above, and let $A \geq 0$.
  Then $|\kappa_s(y)| \leq 1$, so that
  by the Mellin formula we have
  \begin{eqnarray*}
    I_s(l,n,x) &\leq& \int_0^\infty h(x y)
    \kappa_f(m y) \kappa_f(n y) y^{-1} \, \frac{d y}{y} \\
    &=& \int_{(A)} h^\wedge(w) x^w
    \int_{\mathbb{R}_+^*} y^{w - 1} \kappa_f(m y) \kappa_f(n y) \,
    \frac{d y}{y} \, \frac{d w}{2 \pi i} \\
    &=& \frac{(\sqrt{m n })^k}{\left( 4 \pi \left( \frac{m +
            n}{2} \right) \right)^{k-1} } \int_{(A)} h^\wedge(w)
    \left( \frac{x}{4 \pi \left( \frac{m + n }{2} \right)}
    \right)^w
    \Gamma(w + k - 1)
    \, \frac{d w}{2 \pi i} \\
    &\ll_{A}&
    \frac{\Gamma(k-1)}{(4 \pi)^{k-1}}
    \sqrt{m n} \left( \frac{\max(m,n)}{x k}
    \right)^{-A}.
  \end{eqnarray*}
  Here we have used the arithmetic
  mean-geometric mean inequality, the well-known bound
  \cite[Ch 7,
  Misc. Ex 44]{MR0178117}
  \[
  \frac{\Gamma(w+k-1)}{\Gamma(k-1)} \ll_{A} (k-1)^A (1 + k^{-1}(1+|w|^2))
  \ll k^A ( 1 + |w|^2)
  \]
  for $\Re(w) = A$, and the rapid decay
  of $h^\wedge$.  The case $A = 0$ gives
  $I_s(l,n,x) \ll_k (4 \pi)^{-k+1} \Gamma(k-1) \sqrt{m n}$, which combined with the case
  that $A$ is a positive integer yields the assertion of the lemma.
\end{proof}
\begin{remark}
  See \cite[Lem 4.3]{PDN-HMQUE}
  and \cite[Cor 4.4]{PDN-HMQUE}
  for a fairly sharp refinement of Lemma \ref{lem:bound-integral}.
\end{remark}

\begin{corollary}\label{cor:bound-sums}
  The shifted sums $S_s(l,x)$ satisfy the upper bound
  \begin{equation}\label{eq:27}
    S_s(l,x) \ll_{\eps}
    \frac{\Gamma(k-1)}{(4 \pi)^{k-1}}
    \frac{x k}{ \log(x k)^{2-\eps} } 
    \prod_{p
      \leq x k}
    \left( 1+ \frac{2|\lambda_f(p)|}{p} \right)
  \end{equation}
  uniformly for $s \in i \mathbb{R} \cup
  (-1/2,1/2)$ and $x \geq 1$.
\end{corollary}
\begin{proof}
  Let us set $X = x k$
  and temporarily denote by $T_f(x,l,\eps)$ the right-hand side
  of \eqref{eq:27} without the factor
  $(4 \pi)^{-k+1} \Gamma(k-1)$.
  By Definition \ref{defn:shifted-sums}
  and Lemma \ref{lem:bound-integral},
  we need only show that
  \begin{equation}\label{eq:6}
    \sum _{\substack{
        n \in \mathbb{N} \\
        m := n + l \in \mathbb{N}
      }
    }
    | \lambda_f(m) \lambda_f(n) |
    \cdot 
    \max \left(
      1, \frac{\max(m,n)}{X}
    \right)^{-A}
    \ll_\eps T_f(x,l,\eps)
  \end{equation}
  for some positive integer $A$.
  Take $A = 2$.
  We may assume that $X = x k \geq 10$.
  By Theorem \ref{thm:holow-sums-2}
  and the Deligne bound $|\lambda_f(p)| \leq 2$, the left
  hand side of \eqref{eq:6} is
  \begin{eqnarray*}
    &\ll_\eps&
    T_f(x,l,\eps)
    \sum_{n=0}^\infty 2^{-n A} 2^n
    \left( \frac{\log(X)}{\log(2^n X)} \right)^{2-\eps}
    \prod_{X < p \leq 2^n X}
    \left(
      1 + \frac{2|\lambda_f(p)|}{p}
    \right)
    \\
    &\ll&
    T_f(x,l,\eps)
    \sum_{n=0}^\infty 2^{-(A-1) n}
    \exp \left( 4 \log \frac{\log (2^n X)}{\log(X)} \right).
  \end{eqnarray*}
  The inner sum converges and is bounded uniformly
  in $X$, so we obtain the desired estimate \eqref{eq:6}.
\end{proof}

\subsection{Bounds for sums of shifted sums}\label{sec:sums-shifted-sums}
We complete the proof of Theorem \ref{thm:holow} by
bounding the sums of shifted sums that arose
in Proposition \ref{prop:reduce-to-sums}.
\begin{lemma}\label{lem:stupid-sums}
  For each $\eps \in (0,1)$ and each squarefree number $q$, we
  have
  \[
  \sum_{d|q}
  \frac{d}{\log(d k)^{2-\eps}}
  \ll 
  \frac{q \log \log(e^e q)}{\log(q k)^{2-\eps}}
  \ll_\eps
  \frac{q}{\log(q k)^{2-2\eps}}.
  \]
\end{lemma}
\begin{proof}
  Suppose that $q$ is the product of $r \geq 1$ primes $q_1 <
  \dotsb < q_r$.  Let $p_1 < \dotsb < p_r$ be the first $r$
  primes, so that $p_i \leq q_i$ for $i = 1, \dotsc, r$.  Define
  $\beta(x) = x / \log(e^e x k)^{2-\eps}$; we have chosen
  this particular definition so that $\beta$ is increasing on
  $\mathbb{R}_{\geq 1}$
  and $\beta(x) \asymp x/\log(x k)^{2 - \eps}$
  for $x \in \mathbb{R}_{\geq 1}$.  The map
  \[
  \mathbb{R}_{\geq 1} \ni x \mapsto \log \beta(e^x) = x -
  (2-\eps)\log(2+x)\]
  is convex, so that
  for each $a = (a_1, \dotsc, a _r) \in \{0,1\}^r$ we have
  \begin{equation}\label{eq:19}
    \frac{
      \beta(q_1^{a_1} \dotsb q _r ^{a _r })
    }{
      \beta(q_1 \dotsb q _r)
    }
    \leq
    \frac{
      \beta(p_1^{a_1} q_2^{a_2} \dotsb q _r ^{a _r })
    }{
      \beta(p_1 q_2 \dotsb q _r)
    }
    \leq
    \frac{
      \beta(p_1^{a_1} p_2^{a_2} q_3^{a_3} \dotsb q _r ^{a _r })
    }{
      \beta(p_1 p_2 q_3 \dotsb q _r)
    }
    \leq \dotsb \leq
    \frac{
      \beta(p_1^{a_1}\dotsb p _r ^{a _r })
    }{
      \beta(p_1 \dotsb p _r)
    }.    
  \end{equation} 
  The prime number theorem implies
  that $\log(p_1 \dotsb p_r) = r \log(r)(1+o(1))$,
  where the notation $o(1)$ refers to the limit
  as $r \rightarrow \infty$;
  we may and shall assume that $r$ is
  sufficiently large (and at least $100$) because the assertion
  of the lemma holds trivially when $q$ has a bounded number of
  prime factors.
  Set $r_0 = \lfloor r / 10 \rfloor$.
  Observe that
  \begin{eqnarray}\label{eq:23}
    p_{r-r_0+1} \dotsb p_r
    &=&
    \exp \bigl(r \log(r) - (r-r_0)
    \log(r-r_0)
    + o(r \log(r))
    \bigr)
    \\ \nonumber
    &=&
    \exp \left(r_0 \log(r) 
      +(r-r_0)
      \log \left(\frac{r}{r - r_0}\right)
      + o(r \log(r))
    \right)
    \\ \nonumber
    &=&
    \exp \left(r_0 \log(r)(1+o(1))
    \right)
    \\ \nonumber
    &\ll& (p_1 \dotsb p_r)^{1/9},
  \end{eqnarray}
  and
  \begin{equation}\label{eq:22}
    \log(p_1 \dotsb p_{r_0})
    = r_0 \log(r_0)(1+ o(1))
    \asymp r \log(r)(1+ o(1))
    = \log(p_1 \dotsb p_r).
  \end{equation}
  Let $\Omega_0$ denote the set of all
  $a \in \{0, 1 \}^r$
  for which $a_1 + \dotsb + a_r \leq  r_0$
  and 
  $\Omega_1$ the set of all
  $a \in \{0, 1 \}^r$
  for which $a_1 + \dotsb + a_r > r_0$,
  so that $\{0, 1\}^r = \Omega_0 \sqcup \Omega_1$.
  Then by \eqref{eq:23} we have
  \begin{equation}\label{eq:20}
    \sum_{a \in \Omega_0}
    \frac{\beta(p_1 ^{a _1 } \dotsb p _r ^{a _r })}{ \beta (p _1 \dotsb p
      _r )}
    \leq 2^r \frac{\beta(p_{r-r_0+1} \dotsb p_r)}{\beta(p_1 \dotsb p_r)}
    \ll 2^r (p_1 \dotsb p_r)^{-7/8}
    \leq \sqrt[8]{2}.
  \end{equation}
  If $a \in \Omega_1$, then
  \eqref{eq:22} implies
  $\beta(p_1 ^{a _1 } \dotsb p _r ^{a _r }) /  \beta ( p _1 \dotsb p _r )
  \asymp p_1 ^{a_1 - 1} \dotsb p _r ^{a _r -1 }$, so that
  \begin{equation}\label{eq:21}
    \sum _{a \in \Omega _1 }  \frac{\beta (p _1  ^{a _1 }
      \dotsb p _r ^{a _r })}{ \beta (p _1 \dotsb p _r )}
    \ll  \sum _{d | p _1 \dotsb p _r }
    \frac{1}{d}
    \leq (1 + o (1) )
    e ^{\gamma } \log \log (p _1 \dotsb p _r )
    \ll \log \log(e^e q).
  \end{equation}
  Since $\beta(x) \asymp x/\log(e x)^{2 - \eps}$
  for $x \in \mathbb{R}_{\geq 1}$,
  it follows from \eqref{eq:19}, \eqref{eq:20}, and \eqref{eq:21}
  that
  \[
  \frac{
    \displaystyle
    \sum_{d|q} \frac{d}{\log(d k)^{2-\eps}}
  }
  {
    \displaystyle
    \frac{q}{\log(q k)^{2-\eps}}
  }
  \asymp
  \sum_{d|q}
  \frac{\beta(d)}{\beta(q)}
  = \sum_{a \in \{0,1\}^r}
  \frac{\beta(q_1^{a_1} \dotsb q_r^{a_r})}{\beta(q_1 \dotsb q_r)}
  \ll \log \log(e^e q),
  \]
  which establishes the lemma.
\end{proof}

\begin{corollary}\label{prop:bound-sums-of-sums}
  Let $Y \geq 1$ with $Y \leq c_1 \log(q k)^{c_2}$
  for some $c_1, c_2 \geq 1$.  Then
  our sum of shifted sums satisfies
  the estimate
  \[
  \sum_{d | q} S_s(d l, d Y)
  \ll_{\eps,c_1,c_2}
  \frac{\Gamma(k-1)}{(4 \pi)^{k-1}}
  \frac{
    q k Y
  }{ \log(q k)^{2-\eps} }
  \prod_{p
    \leq q k}
  \left( 1+ \frac{2 |\lambda_f(p)|}{p} \right),
  \]
  uniformly for $s \in i \mathbb{R} \cup
  (-1/2,1/2)$ and $x \geq 1$.
\end{corollary}
\begin{proof}
  By Corollary \ref{cor:bound-sums},
  we have
  \begin{equation}\label{eq:18}
    \sum_{d | q} S_s(d l, d Y)
    \ll_{\eps}
    \frac{\Gamma(k-1)}{(4 \pi)^{k-1}}
  Y
    \left( 
      \prod_{p
        \leq q k Y}
      \left( 1+ 2 \frac{|\lambda_f(p)|}{p} \right)
    \right)
    \sum_{d | q} \frac{ d k}{\log(d k)^{2 - \eps}}.
  \end{equation}
  By the Deligne bound $|\lambda_f(p)| \leq 2$, the part of the
  product in \eqref{eq:18} taken over $q k < p \leq q k Y$ is $\ll
  \log(e Y)^4 \ll_{c_1, c_2} \log \log(e^e q k)^4$.  The claim now follows
  from Lemma \ref{lem:stupid-sums}.
\end{proof}

\begin{lemma}\label{lem:cspsrs}
  Let $\eps > 0$, $Y \geq 1$.
  If $\phi$ is a normalized Maass eigencuspform, then
  \[
  \sum _{
    0 < |l| < Y^{1+\eps}
  }
  \frac{  \left\lvert \lambda_\phi(l) 
    \right\rvert
  }
  {
    \sqrt{|l|}
  }
  \ll_{\phi,\eps} Y^{1/2+2 \eps},
  \]
  where (as indicated) the implied constant may depend upon $\phi$.
  On the other hand, if $t \in \mathbb{R}$, then
  \[
  \sum _{
    0 < |l| < Y^{1+\eps}
  }
  \frac{  \left\lvert \lambda_{i t}(l) 
    \right\rvert
  }
  {
    \sqrt{|l|}
  }
  \ll_{\eps} Y^{1/2+2\eps},
  \]
  where the implied constant does not depend upon $t$.
\end{lemma}
\begin{proof}
  Follows from the
  Cauchy-Schwarz inequality, partial summation,
  the Rankin-Selberg bound
  (\ref{eq:rs-bound}) for $\lambda_\phi$
  and the uniform
  bound $|\lambda_{i t}(l)| \leq \tau(l)$ for $\lambda_{it}$.
\end{proof}

\begin{proof}[Proof of Theorem \ref{thm:holow}]
  Suppose that $\phi$ is a normalized Maass eigencuspform
  of eigenvalue $\frac{1}{4} + r^2$.
  By Proposition \ref{prop:reduce-to-sums}, we have
  \begin{equation}\label{eq:25}
    \frac{\mu_f(\phi)}{\mu_f(1)}
    = \frac{1}{Y \mu_f(1)}
    \sum_{
      0 < |l| < Y^{1+\eps}
    }
    \frac{\lambda_\phi(l)}{\sqrt{|l|}}
    \sum_{d|q} S_{i r}(d l, d Y)
    + O_{\phi,\eps}(Y^{-1/2}).
  \end{equation}
  Recall from (\ref{eq:2})
  that
  \begin{equation*}
    \mu_f(1) \asymp q \frac{\Gamma(k-1)}{(4 \pi)^{k-1}}  L(\ad f, 1)
  \end{equation*}
  and recall the definition \eqref{eq:12}
  of $M_f(q k)$.
  We shall ultimately choose $Y \ll \log(q k)^{O(1)}$,
  so
  Corollary \ref{prop:bound-sums-of-sums}
  gives the bound
  \begin{equation}\label{eq:28}
    \frac{1}{Y \mu_f(1)}
    \sum_{d | q} S_{i r}(d l, d Y)
    \ll_{\eps}
    \log(q k)^{\eps}
    M_f(q k).
  \end{equation}
  By \eqref{eq:28} and Lemma \ref{lem:cspsrs}
  applied to \eqref{eq:25},
  we find that
  \begin{eqnarray*}
    \frac{\mu_f(\phi)}{\mu_f(1)}
    &\ll_{\phi,\eps}&
    \log(q k)^{\eps}
    M_f(q k)
    \sum_{0 <
      |l| < Y^{1+\eps}
    }
    \frac{\left\lvert \lambda_\phi(l)
      \right\rvert}{\sqrt{|l|}}
    + Y^{-1/2} \\
    &\ll_{\phi,\eps}&
    Y^{1/2+2\eps}
    \log(q k)^{\eps}
    M_f(q k)
    + Y^{-1/2}.
  \end{eqnarray*}
  Choosing $Y = \max(1,M_f(q k)^{-1}) \ll \log(q k)^{O(1)}$ gives the cuspidal
  case of the theorem.

  Suppose now that $\phi = E(\Psi,\cdot)$ is an incomplete
  Eisenstein series.  Proposition \ref{prop:reduce-to-sums},
  Corollary \ref{prop:bound-sums-of-sums} and Lemma
  \ref{lem:cspsrs} show, as in the cuspidal case, that
  \begin{eqnarray*}
    \frac{\mu_f(\phi)}{\mu_f(1)}
    - \frac{\mu(\phi)}{\mu(1)}
    &\ll_{\phi,\eps }&
    Y^{1/2+2 \eps} \log(q k)^\eps M_f(q k)
    \int _{\mathbb{R}}
    \left\lvert \frac{\Psi ^\wedge (\tfrac{1}{2} + i
        t)}{ \xi (1 + 2 i t)}  \right\rvert
    \, d t
    + \frac{1 + R_f(q k)}{Y^{1/2}} \\
    &\ll_{\phi}& Y^{1/2+2\eps}
    \log(q k)^\eps M_f(q k) 
    + \frac{1 + R_f(q k)}{Y^{1/2}}.
  \end{eqnarray*}
  The same choice of $Y$ as above completes the proof.
\end{proof}

\section{An extension of Watson's formula}\label{sec:an-extension-watsons}
Watson \cite{watson-2008}, building on earlier work of Garrett
\cite{MR881269}, Piatetski-Shapiro and Rallis \cite{MR911357},
Harris and Kudla \cite{harris-kudla-1991}, and Gross and Kudla
\cite{MR1145805}, proved a beautiful formula relating the
integral of the product of three modular forms to the central
value of their triple product $L$-function.  Unfortunately, Watson's
formula applies only to triples of newforms having the
\emph{same} squarefree level.  In \S
\ref{sec:proof-theor-refthm:1} we shall refer only to the
statement of the following extension of Watson's formula to the
case of interest, not the details of its proof.
\begin{theorem}\label{thm:watson-ext}
  Let $\phi$ be a Maass eigencuspform of level $1$
  and $f$ a holomorphic newform of squarefree level $q$,
  as in \S \ref{sec:modular-forms-their}.
  Then
  \[
  \frac{
    \left\lvert \int_{\Gamma_0(q) \backslash \mathbb{H}}
      \phi(z) |f|^2(z) y^k \, \frac{d x \, d y}{y^2}
    \right\rvert^2
  }{
    \int_{\Gamma \backslash \mathbb{H}}
    |\phi|^2(z) y^k \, \frac{d x \, d y}{y^2}
    \left( \int_{\Gamma_0(q) \backslash \mathbb{H}}
      |f|^2(z) y^k \, \frac{d x \, d y}{y^2}
    \right)^2
  }
  = \frac{1}{8 q}
  \frac{\Lambda(\phi \times f \times f,\tfrac{1}{2})}{
    \Lambda(\ad \phi,1) \Lambda(\ad f,1)^2}.
  \]
  The $L$-functions $L(\dotsb) = \prod_p L_p (\dotsb)$ and their completions
  $\Lambda(\dotsb) = L_\infty(\dotsb) L(\dotsb) = \prod_v L_v(\dotsb)$
  are as in \cite[\S 3]{watson-2008}.
\end{theorem}
\begin{remark}\label{rmk:an-extension-watsons-1}
  For simplicity, we have stated Theorem \ref{thm:watson-ext}
  only in the special case that we need it, but our calculations
  (Lemma \ref{lem:p-adic-integral}) lead to a more general
  formula.  Let $\psi_j$ ($j=1,2,3)$ be newforms of weight $k_j$
  and level $q_j$.  We allow the possibility $k_j = 0$, in which
  case we require that $\psi_j$ be an even or odd Maass
  eigencuspform.  If $k_1 + k_2 + k_3 \neq 0$ or some prime $p$
  divides exactly one of the $q_j$, then
  it is straightforward to see that $\int \psi_1 \psi_2
  \psi_3 = 0$.  Otherwise $k_1+k_2+k_3=0$ and each prime divides
  the $q_j$ either $0,2$ or $3$ times, so one can read off from
  Watson \cite[Theorem 3]{watson-2008}, Ichino \cite{MR2449948} and Lemma
  \ref{lem:p-adic-integral} the identity
  \begin{equation}
    \frac{
      \left\lvert  \int_X \psi_1 \psi_2 \psi_3
      \right\rvert^2
    }
    {
      \prod \int_X |\psi_j|^2
    }
    = \frac{1}{8}
    \frac{\Lambda(\tfrac{1}{2},\psi_1 \times \psi_2 \times \psi_3) }{ \prod \Lambda (1, \ad
      \psi_j)}
    \prod_v c_v 
  \end{equation}
  where $X = \varprojlim \Gamma_0(q) \backslash \mathbb{H}$ with
  $\vol(X) := \vol(\Gamma_0(1) \backslash \mathbb{H}) = \pi /
  3$, $c_\infty$ is $Q_\infty \in \{0,1,2\}$ from \cite[Theorem
  3]{watson-2008}, $c_p = 1$ if $p$ divides none of the $q_j$,
  $c_p = p^{-1}$ if $p$ divides exactly two of the $q_j$, and
  $c_p = p^{-1}(1+p^{-1})(1+\eps_p)$ if $p$ divides all of the
  $q_j$ with $-\eps_p$ the product of the Atkin-Lehner
  eigenvalues for the $\psi_j$ at $p$ as in \cite[Theorem 3]{watson-2008}.
\end{remark}

Watson proved his formula only for three forms of the same
squarefree level because Gross and Kudla \cite{MR1145805}
evaluated the $p$-adic zeta integrals of Harris and Kudla
\cite{harris-kudla-1991} only when (the factorizable automorphic
representations generated by) the three forms are special at
$p$; Harris and Kudla had already considered the case that all
three forms are spherical at $p$.  Ichino \cite{MR2449948}
showed that the local zeta integrals of Harris and Kudla
are equal to simpler integrals over the group
$\PGL(2,\mathbb{Q}_p)$.  Ichino and Ikeda \cite[\S 7, \S
12]{MR2585578} computed these simpler integrals when all three
forms are special at $p$.  Since we are interested in the
integral of $\phi |f|^2$ when $\phi$ has level $1$ and $f$ has
squarefree level $q$, we must consider the case that two
representations are special and one is spherical.  We remark in
passing that B\"{o}cherer and Schulze-Pillot \cite{BoSP96}
considered similar problems for modular forms on definite
rational quaternion algebras in the classical language, but
their results are not directly applicable here.

To state (a special case of) Ichino's result, we introduce
some notation.
In what follows,
$v$ denotes a place of $\mathbb{Q}$
and
$p$ a prime number.
Let $G = \PGL(2)/\mathbb{Q}$,
$G_v = G(\mathbb{Q}_v)$,
$K_\infty = \SO(2) / \{\pm 1\}$,
$K_p = G(\mathbb{Z}_p)$,
and $G_\mathbb{A} = G(\mathbb{A}) = \prod_v' G_v$,
where $\mathbb{A} = \prod_v ' \mathbb{Q}_v$ is the adele ring of $\mathbb{Q}$.
Regard $\phi$ and $f$ as pure tensors
$\phi = \bigotimes \phi_v$ and $f = \bigotimes f_v$
in (factorizable) cuspidal automorphic representations
$\pi_\phi = \bigotimes \pi_{\phi,v}$ and $\pi_f = \bigotimes
\pi_{f,v}$
of $G_\mathbb{A} = \prod ' G_v$.
Set
$\bar{f}_v =
\left(\begin{smallmatrix}
    - 1&\\
    &1
  \end{smallmatrix}
\right) \cdot f_v$
and 
$\bar{f} = \bigotimes \bar{f}_v$.
Then $f_p = \bar{f}_p$ for all (finite) primes $p$.
Although the vectors $\phi_v$ and $f_v$ are defined only up to a
nonzero scalar multiple,
the matrix coefficients
\[
\Phi_{\phi,v}(g_v)
= \frac{\langle g_v \cdot \phi_v, \phi_v \rangle}{
  \langle \phi_v, \phi_v \rangle} ,
\quad 
\Phi_{f,v}(g_v)
= \frac{\langle g_v \cdot f_v, f_v \rangle}{
  \langle f_v, f_v \rangle} ,
\quad \Phi_{\bar{f},v}(g_v)
= \frac{\langle g_v \cdot \bar{f}_v, \bar{f}_v \rangle}{
  \langle \bar{f}_v, \bar{f}_v \rangle} 
\]
are well-defined; here $g_v$ belongs to $G_v$
and $\langle ,
\rangle_v$
denotes the (unique up to a scalar) $G_v$-invariant Hermitian pairings
on the irreducible admissible self-contragredient representations $\pi_{\phi,v}$
and $\pi_{f, v}$.
Let $d g_v$ denote the Haar measure on the group $G_v$
with respect to which $\vol(K_v) = 1$.
Define the local integrals
\[
I_v = \int_{G_v}
\Phi_{\phi,v}(g_v) \Phi_{f,v}(g_v) \Phi_{\bar{f},v}(g_v) \, d g_v
\]
and the normalized local integrals
\begin{equation}\label{eq:5}
  \tilde{I}_v =
  \left( \frac{\zeta_v(2)^3}{\zeta_v(2)}
    \frac{L_v(\tfrac{1}{2}, \phi \times f \times f)}{
      L_v(1, \ad \phi) L_v(1, \ad f)^2} \right)^{-1}
  I_v.
\end{equation}
\begin{theorem}[Ichino]\label{thm:ichino}
  We have $\tilde{I}_v = 1$ for all but finitely many places
  $v$,
  and
  \[
  \frac{
    \left\lvert 
      \int _{\Gamma _0 (q) \backslash \mathbb{H} } \phi |f|^2 y^k \,
      \frac{d x \, d y}{y^2}\right\rvert^2}{
    \int _{\Gamma \backslash \mathbb{H} } |\phi|^2 \, \frac{d x \, d
      y}{y^2}
    \left(  \int _{\Gamma _0 (q)
        \backslash \mathbb{H} } |f|^2 y^k \, \frac{d x \, d y}{y^2}
    \right)^2
  }=
  \frac{1}{8}
  \frac{\Lambda(\tfrac{1}{2}, \phi \times f \times f)}{
    \Lambda(1, \ad \phi) \Lambda(1, \ad f)^2}
  \prod_v \tilde{I}_v.
  \]
\end{theorem}
\begin{proof}
  See \cite[Theorem 1.1, Remark 1.3]{MR2449948}.  We have taken
  into account the relation between classical modular
  forms
  and automorphic forms on the adele group $G_\mathbb{A}$ (see Gelbart
  \cite{MR0379375}) and the comparison (see
  for instance Vign{\'e}ras \cite[\S III.2]{MR580949}) between the Poincar\'{e} measure on the
  upper half-plane and the Tamagawa measure
  on $G_\mathbb{A}$.
\end{proof}

We know by work of Harris and Kudla
\cite{harris-kudla-1991}, Gross and Kudla \cite{MR1145805},
Watson \cite{watson-2008}, Ichino \cite{MR2585578}, and Ichino
and Ikeda \cite{MR2585578} that $\tilde{I}_\infty = 1$ and
$\tilde{I}_p = 1$ for all primes $p$ that do not divide the
level $q$.  
We contribute the following computation,
with which we deduce
Theorem \ref{thm:watson-ext}
from Theorem \ref{thm:ichino}.
\begin{lemma}\label{lem:p-adic-integral}
  Let $p$ be a prime divisor of the squarefree level $q$.  Then
  $\tilde{I}_p = 1/p$.
\end{lemma}
Before embarking on the proof,
let us introduce some notation
and recall formulas for the matrix coefficients
$\Phi_{\phi,p}$ and $\Phi_{f,p}$.
Let $G_p = \PGL_2(\mathbb{Q}_p)$,
let
$K_p = \PGL_2(\mathbb{Z}_p)$,
and let $A_p$ be the subgroup of diagonal matrices in $G_p$.
Recall the Cartan decomposition $G_p = K_p A_p K_p$.
For $y \in \mathbb{Q}_p^*$ we write
$a(y) = \left(
  \begin{smallmatrix}
    y&\\
    &1
  \end{smallmatrix}
\right) \in A_p$.

The representation $\pi_{\phi,p}$ is unramified principal series
with Satake parameters $\alpha_{\phi}(p)$ and $\beta_\phi(p)$;
for clarity we write simply $\alpha = \alpha_{\phi}(p)$ and
$\beta = \beta_{\phi}(p)$.  The vector $\phi_p$
lies on the unique $K_p$-fixed line in $\pi_{\phi,p}$.
The matrix coefficient $\Phi_{f,p}$
is bi-$K_p$-invariant, so by the Cartan decomposition we need
only specify $\Phi_{\phi ,p}(a(p^m))$ for $m \geq 0$, which is
given by the Macdonald formula \cite[Theorem 4.6.6]{MR1431508}
\begin{equation}\label{eq:16}
  \Phi_{\phi ,p}(a(p^m))
  =
  \frac{1}{1 + p ^{-1} }
  p ^{- m / 2}
  \left[ \alpha^m \frac{1 - p^{-1} \frac{\beta }{\alpha }}{
      1 - \frac{\beta }{\alpha }} + \beta ^m \frac{1 - p ^{-1}
      \frac{\alpha }{\beta }}{ 1 - \frac{\alpha }{\beta }}
  \right].
\end{equation}

The representation $\pi_{f,p}$ is an unramified quadratic twist
of the Steinberg representation of $G_p$.  The vector $f_p$ lies
on the unique $I_p$-fixed line in $\pi_{f,p}$, where $I_p$ is
the Iwahori subgroup of $K_p$ consisting of matrices that are
upper-triangular mod $p$.  Thus to determine $\Phi_{f,p}$, we
need only specify the values it takes on representatives for the
double coset space $I_p \backslash G_p / I_p$, whose structure
we now recall following \cite[\S 7]{MR0342495} (see also
\cite[\S 7]{MR2585578} for a similar discussion).  Define the
elements
\[
w_1 = 
\begin{pmatrix}
  &1\\
  1&
\end{pmatrix}
,
\quad 
w_2 = 
\begin{pmatrix}
  &p ^{-1} \\
  p&
\end{pmatrix}
,
\quad 
\omega = 
\begin{pmatrix}
  &1\\
  p&
\end{pmatrix}
\]
of $G_p$.  Note that since $G_p = \PGL_2(\mathbb{Q}_p)$, we have
$w_1^2 = w_2^2 = \omega^2 = 1$.  For $w$ in the group $W_a =
\langle w_1, w_2 \rangle$ generated by $w_1$ and $w_2$, let
$\lambda(w)$ be the length of the shortest string expressing $w$
in the alphabet $\{w_1, w_2\}$, so that $\lambda(w_1) =
\lambda(w_2) = 1$.  Extend $\lambda$ to the group $\tilde{W}
= \langle w_1, w_2, \omega  \rangle$,
which is the semidirect product of $W_a$ by the group of order
$2$
generated by $\omega$,
via the formula
$\lambda(\omega^i w) = \lambda(w)$ when $w \in W_a$, so that in
particular $\lambda(\omega) = 0$.  We have a Bruhat
decomposition $G_p = \bigsqcup_{w \in \tilde{W}} I_p w I_p$;
unwinding the definitions, this reads more concretely
as
\[
G_p =
\left(
  \bigsqcup_{n \in \mathbb{Z}}
I_p \begin{pmatrix}
  p^n &  \\
   & 1
 \end{pmatrix} I_p
\right)
\bigsqcup
\left(
  \bigsqcup_{n \in \mathbb{Z}}
  I_p
  w_1
  \begin{pmatrix}
  p^n &  \\
   & 1
 \end{pmatrix} I_p
\right),
\]
but we shall not adopt this perspective.
With
our normalization of measures we have $\vol(I_p w I_p) = (p +
1)^{-1} p^{\lambda(w)}$.  Suppose temporarily that $\pi_{f,p}$
is (the \emph{trivial} twist of) the Steinberg representation.
The matrix coefficient $\Phi_{f,p}$ is bi-$I_p$-invariant and
given by
\[
\Phi_{f,p}(\omega^j w) = (-1)^j (-p^{-1})^{\lambda(w)}\] for
all $j \in \{0,1\}$ and $w \in W_a$.
In particular
\begin{equation}\label{eq:15}
  \Phi_{f,p}(\omega^j w)^2 = p^{-2 \lambda(w)}.
\end{equation}
In the general case that $\pi_{f,p}$ is a possibly nontrivial
unramified quadratic twist of Steinberg, the formula
\eqref{eq:15} for the \emph{squared} matrix coefficient
still holds.

\begin{proof}[Proof of Lemma \ref{lem:p-adic-integral}]
  Having recalled the formulas above, we see that
  \begin{eqnarray}\label{eq:17}
    I_p &=& \int_{G_p}
    \Phi_{\phi,p}(g) \Phi_{f,p}(g)^2 \, d g
    = \sum_{w \in \tilde{W}}
    \vol(I_p w I_p)
    \Phi_{\phi,p}(w)
    p^{-2 \lambda(w)} \\ \nonumber
    &=& (p+1)^{-1}
    \sum_{w \in \tilde{W}}
    \Phi_{\phi,p}(w)
    p^{-\lambda(w)},
  \end{eqnarray}
  where $\Phi_{\phi,p}$ is given by \eqref{eq:16}.
  The evaluation of the Poincar\'{e} series
  \begin{equation}\label{eq:13}
    \sum_{w \in \tilde{W}} t^{\lambda(w)}
    = 2 \frac{1+t}{1-t},
  \end{equation}
  where $t$ is an indeterminate, is asserted and used in \cite[\S
  7]{MR2585578}, but we need a finer result
  here.  For $w \in \tilde{W}$ let us write $\mu(w)$ for the
  unique nonnegative integer with the property that
  $K_p w K_p = K_p a(p^{\mu(w)}) K_p$.
  Then we claim that for indeterminates $x,t$ we have
  the relation of formal power series
  \begin{equation}\label{eq:14}
    \sum_{w \in \tilde{W}}
    x^{\mu(w)} t^{\lambda(w)}
    = \frac{(1+x)(1+t)}{1-x t}.
  \end{equation}
  Note that we recover \eqref{eq:13}
  upon taking $x = 1$.
  To prove \eqref{eq:14},
  observe that
  since $\omega w_1 = w_2 \omega$ and $\omega^2 = 1$,
  every element $w$ of $\tilde{W}$
  is of the form
  $u_{abn} = \omega^a (w_1 w_2)^n w_1^b$
  or
  $v_{abn} = \omega^a (w_2 w_1)^n w_2^b$
  for some $a \in \{0,1\}$, $b \in \{0,1\}$, and $n \in
  \mathbb{Z}_{\geq 0}$.
  Computing $u_{a b n}$ and $v_{a b n}$ explicitly to be
  \[
  u_{0 0 n}
  = \left(
    \begin{matrix}
      p^n&\\
      &p^{-n}
    \end{matrix}
  \right),
  \quad 
  u_{0 1 n}
  = \left(
    \begin{matrix}
      &p^n\\
      p^{-n}&
    \end{matrix}
  \right),
  \]
  \[
  u_{1 0 n}
  = \left(
    \begin{matrix}
      &p^{-n}\\
      p^{n+1}&
    \end{matrix}
  \right),
  \quad 
  u_{1 1 n}
  = \left(
    \begin{matrix}
      p^{-n}&\\
      &p^{n+1}
    \end{matrix}
  \right),
  \]
  \[
  v_{0 0 n}
  = \left(
    \begin{matrix}
      p^{-n}&\\
      &p^{n}
    \end{matrix}
  \right),
  \quad 
  v_{0 1 n}
  = \left(
    \begin{matrix}
      &p^{-n-1}\\
      p^{n+1}&
    \end{matrix}
  \right),
  \]
  \[
  v_{1 0 n}
  = \left(
    \begin{matrix}
      &p^{n}\\
      p^{1-n}&
    \end{matrix}
  \right),
  \quad 
  v_{1 1 n}
  = \left(
    \begin{matrix}
      p^{n+1}&\\
      &p^{-n}
    \end{matrix}
  \right),
  \]
  we see that this parametrization of $\widetilde{W}$ is unique
  except that $u_{a00} = v_{a 00}$ for each $a \in
  \{0,1\}$; furthermore, we can read off that
  $\mu(u_{a b n}) = 2 n + a $, that $\mu(v_{a
    b n}) = 2(n+b) - a$, and that $\lambda(u_{a b n}) = \lambda(v_{a
    b n}) = 2 n + b$.  Thus
  \begin{align*}
    \sum_{w \in \tilde{W}} x^{\mu(w)} t^{\lambda(w)}
    &=
    (1+x) +
    \mathop{    \sum_{b =0,1}  \sum_{n \geq 0}    }_{2n+b > 0}
    t^{2 n + b}
    \sum_{a =0,1}
    \left( x^{2 n + a} + x^{2(n+b)-a} \right) \\
    &=
    (1+x) +
    \mathop{    \sum_{b =0,1}  \sum_{n \geq 0}    }_{2n+b > 0}
    t^{2 n + b}
    x^{2 n + b -1}
    \sum_{a =0,1}
    \left( x^{1 + a-b} + x^{1+b-a} \right) \\
    &=
    (1+x) + (1+x)^2 \sum_{m > 0} t^m x^{m-1},
  \end{align*}
  from which \eqref{eq:14} follows upon summing the geometric series.
  We now combine \eqref{eq:16}, \eqref{eq:17}
  and \eqref{eq:14}, noting that the series converge
  because $|\alpha| < p^{1/2}$
  and $|\beta| < p^{1/2}$;
  the contributions to the formula \eqref{eq:17} for $I_p$
  of the two terms in the formula \eqref{eq:16}
  for $\Phi_{\phi,p}(a(p^m))$ are respectively
  \[
  (p+1)^{-1} (1+p^{-1})^{-1}
  \frac{
    1 - p ^{-1} \frac{\beta }{ \alpha }
  }
  {
    1 - \frac{\beta}{\alpha }
  }
  \frac{
    (1 + p ^{- 1/2} \alpha )(1+p^{-1})
  }
  {
    1 - p^{-3/2} \alpha
  },
  \]
  and
  \[
  (p+1)^{-1} (1+p^{-1})^{-1}
  \frac{
    1 - p ^{-1} \frac{\alpha }{ \beta }
  }
  {
    1 - \frac{\alpha}{\beta }
  }
  \frac{
    (1 + p ^{- 1/2} \beta )(1+p^{-1})
  }
  {
    1 - p^{-3/2} \beta
  }.
  \]
  Summing these fractions by cross-multiplication and then simplifying, we obtain
  \[
  I_p
  = p^{-1} ( 1-p^{-1})
  \frac{ (1 + \alpha p ^{- 1/2 }) (1 + \beta p ^{- 1 / 2})}{
    (1 - \alpha p ^{- 3 / 2}) (1 - \beta p ^{- 3 / 2})}.
  \]
  Recall the definition \eqref{eq:5} of
  $\tilde{I}_p$.
  The local $L$-factors are given by  (see \cite[\S 3.1]{watson-2008})
  \[
  L_p(1,\ad f) = \zeta_p(2),
  \quad 
  L_p(1,\ad \phi) = [(1-\alpha ^2 p^{-1})(1 - p^{-1})(1 -
  \beta^2 p^{-1})] ^{-1},
  \]
  \[
  L_p(\tfrac{1}{2},\phi \times f \times f)
  = [(1-\alpha p^{-1/2})(1-\beta p^{-1/2})
  (1-\alpha p^{-3/2})(1-\beta p^{-3/2})]^{-1},
  \]
  thus the normalized local integral $\tilde{I}_p$
  is
  \[
  \tilde{I}_p
  =
  p^{-1} ( 1- p^{-1})
  \frac{
    (1- \alpha p^{-1/2})(1 - \beta p^{-1/2}) (1 + \alpha p ^{- 1/2 }) (1 + \beta p ^{- 1 / 2})
  }{(1 - \alpha^2 p^{-1}) (1 - p^{-1}) ( 1 - \beta^2 p^{-1})
  }
  = p^{-1},
  \]
  as asserted.
\end{proof}

\section{Proof of Theorem
  \ref{thm:1}}\label{sec:proof-theor-refthm:1}
We combine Theorem \ref{thm:holow} and Theorem
\ref{thm:watson-ext} with Soundararajan's weak subconvex bounds
\cite{MR2680497} to complete the proof of Theorem
\ref{thm:1}.
Fix a positive even integer $k$.
Let $f$ be a newform of weight $k$ and squarefree level $q$.
Fix a Maass eigencuspform or incomplete Eisenstein series
$\phi$.
We will show that the ``discrepancy''
\[
D_f(\phi) := \frac{\mu_f(\phi)}{\mu_f(1)}
- \frac{\mu(\phi)}{\mu(1)}
\]
tends to $0$ as
$q k \rightarrow \infty$,
thereby fulfilling the criterion of Lemma \ref{prop:3},
by combining the complementary estimates
for $D_f(\phi)$
provided below by Proposition \ref{prop:holow-refined-level}
and Proposition \ref{prop:sound-level}.
\begin{lemma}\label{lem:boundz}
  The quantities $M_f(x)$ and $R_f(x)$ (\ref{eq:12}) appearing in
  the statement of Theorem \ref{thm:holow} satisfy
  the estimates
  \[
  M_f(q k) \ll_{\eps} \log(q k)^{1/6+\eps} L(\ad f,1)^{1/2},
  \quad 
  R_f(q k) \ll_{\eps} \frac{\log(q k)^{-1+\eps}}{L(\ad f,1)}
  \ll \log(q k)^{\eps}.
  \]
\end{lemma}
\begin{proof}
  The bound for $M_f(q k)$ follows from the proof of \cite[Lemma
  3]{MR2680499}
  with ``$k$'' replaced by ``$q k$,'' noting that $\lambda_f(p)^2
  \leq 1 + \lambda_f(p^2)$ for \emph{all} primes $p$.  The bound
  for $R_f(q k)$ follows from the arguments of \cite[Example
  1]{MR2680497}, \cite[Lemma
  1]{MR2680499} with ``$k$'' replaced by ``$q k$'' and the lower bound (\ref{eq:hlghl})
  for $L(\ad f, 1)$.
\end{proof}
\begin{proposition}\label{prop:holow-refined-level}
  We have $D_f(\phi) \ll_{\phi,\eps} \log(q k)^{1/12+\eps} L(\ad f, 1)^{1/4}$.
\end{proposition}
\begin{proof}
  Follows immediately from Theorem \ref{thm:holow}
  and Lemma \ref{lem:boundz}.
\end{proof}

\begin{proposition}\label{prop:sound-level}~
  We have $D_f(\phi) \ll_{\phi,\eps} \log(q k)^{-\delta +\eps} L(\ad
  f,1)^{-1}$,
  where $\delta = 1/2$ if $\phi$ is a Maass eigencuspform
  and $\delta = 1$ if $\phi$ is an incomplete Eisenstein series.
\end{proposition}
\begin{proof}
  If $\phi$ is a Maass eigencuspform,
  then the analytic conductor of $\phi \times f \times f$ is
  $\asymp (q k)^4$, so Theorem \ref{thm:watson-ext} and
  the arguments of Soundararajan \cite[Example
  2]{MR2680497} with ``$k$'' replaced by ``$q k$''
  show that
  \[
  \left\lvert \frac{\mu _f (\phi) }{ \mu _f (1) } \right\rvert
  ^2 
  \ll_{\phi} \frac{L(\phi \times f \times f,\tfrac{1}{2})}{ q k \cdot
    L(\ad f,1)^2}
  \ll_{\eps} \frac{1}{\log(q k)^{1-\eps} L(\ad f,1)^2}.
  \]
  If $\phi = E(\Psi,\cdot)$
  is an incomplete Eisenstein series,
  then the unfolding method
  as in Lemma \ref{lem:3} and the bound
  for $R_f(q)$ given by Lemma \ref{lem:boundz} show that
  \begin{align*}
    \frac{\mu_f(\phi)}{\mu_f(1)}
    - \frac{\mu(\phi)}{\mu(1)}
    &= \frac{2 \pi^2}{q}
    \int_{(1/2)} \Psi^\wedge(s)
    \left( \frac{q}{4 \pi } \right)^s
    \frac{\Gamma(s + k-1)}{\Gamma(k)}
    \frac{\zeta(s)}{\zeta(2 s)}
    \frac{L(\ad f,s)}{L(\ad f,1)}
    \, \frac{d s}{2 \pi i} \\
    &\ll_{\phi} R_f(q k)
    \ll_{\eps} \frac{\log(q k)^{-1+\eps}}{L(\ad f,1)}.
  \end{align*}
\end{proof}

\begin{proof}[Proof of Theorem \ref{thm:1}]
  By Propositions \ref{prop:holow-refined-level} and
  \ref{prop:sound-level}, there exists $\delta \in \{1/2,1\}$
  such that
  \[
  D_f(\phi) \ll_{\phi,\eps} \min \left( \log(q k) ^{-\delta + \eps}
    L(\ad f,1)^{-1}, \log(q k)^{1/12+\eps} L(\ad f, 1)^{1/4}
  \right);
  \]
  it follows by the argument of \cite[\S
  3]{MR2680499} with ``$k$'' replaced by
  ``$q k$'' that $D_f(\phi) \rightarrow 0$ as $q k \rightarrow
  \infty$.
\end{proof}

\bibliography{refs2}{}
\bibliographystyle{plain}
\end{document}